\documentclass[11pt, letterpaper]{amsart}
\usepackage[left=1in,right=1in,bottom=1in,top=1in]{geometry}

\usepackage{amsmath,amsthm,amssymb}
\usepackage{mathrsfs}
\usepackage[all,cmtip]{xy}
\usepackage{cite}
\usepackage{graphicx}
\usepackage{hyperref}
\usepackage[shortlabels]{enumitem}
\usepackage{multicol}

\newcommand{\RR}{\mathbf R}

\newcommand{\ZZ}{\mathbf Z}
\newcommand{\QQ}{\mathbf Q}
\newcommand{\QQbar}{\overline{\QQ}}
\newcommand{\kbar}{\overline{k}}

\newcommand{\FF}{\mathbf F}
\newcommand{\FFbar}{\overline{\FF}}

\newcommand{\gl}{\mathfrak{gl}}

\newcommand{\g}{\mathfrak{g}}

\newcommand{\m} {\mathfrak m}

\DeclareFontFamily{OT1}{rsfs}{}
\DeclareFontShape{OT1}{rsfs}{n}{it}{<-> rsfs10}{}
\DeclareMathAlphabet{\mathscr}{OT1}{rsfs}{n}{it}

\renewcommand{\O}{\mathcal{O}}

\newcommand{\on}[1]{\operatorname{#1}}
\newcommand{\Hom}{\on{Hom}}

\newcommand \tensor[1] {\otimes_{#1}}

\renewcommand{\Im}{\on{Im}}

\newcommand{\Gal}{\on{Gal}}

\newcommand{\End}{\on{End}}

\newcommand{\Lie}{\on{Lie}}

\newcommand{\GL}{\on{GL}}
\newcommand{\SL}{\on{SL}}

\newcommand{\into}{\hookrightarrow}

\theoremstyle{plain}
\newtheorem{lem}{Lemma}
\newtheorem{thm}[lem]{Theorem}
\newtheorem{prop}[lem]{Proposition}
\newtheorem{cor}[lem]{Corollary}

\newtheorem{fact}[lem]{Fact}

\theoremstyle{definition}
\newtheorem{defn}[lem]{Definition}
\newtheorem{example}[lem]{Example}
\newtheorem{remark}[lem]{Remark}

\newcommand{\ttm}[4]{\begin{pmatrix}
#1 & #2 \\
#3 & #4
\end{pmatrix}}

\numberwithin{equation}{section}
\numberwithin{lem}{section}

\newcommand{\ad}{\on{ad}}

\newcommand{\GSp}{\on{GSp}}
\newcommand{\GO}{\on{GO}}

\newcommand{\Gm}{\mathbf{G}_m}

\newcommand{\Id}{\on{Id}}

\newcommand{\Fil}{\on{Fil}}

\newcommand{\rhobar}{{\overline{\rho}}}
\newcommand{\adrho}{\on{ad}(\rhobar)}
\newcommand{\adzerorho}{\on{ad}^0(\rhobar)}

\newcommand{\chibar}{\overline{\chi}}
\newcommand{\nubar}{\overline{\nu}}
\newcommand{\Vbar}{\overline{V}}
\newcommand{\Mbar}{\overline{M}}

\newcommand{\onto}{\twoheadrightarrow}
\newcommand{\inverselimit}{\varprojlim}
\newcommand{\directlimit}{\varinjlim}
\newcommand{\Frob}{\on{Frob}}

\newcommand{\res}{\on{res}}

\newcommand{\C}{\mathcal{C}}

\newcommand{\D}{\mathcal{D}}
\newcommand{\Rep}{\on{Rep}}

\newcommand{\MF}{\on{MF}}

\renewcommand{\t}{\mathfrak{t}}

\newcommand{\zfrak}{\mathfrak{z}}

\newcommand{\Ghat}{\widehat{G}}

\newcommand{\Dram}{\D^{\on{ram}}}

\newcommand{\cris}{{\on{cris}}}
\newcommand{\Acris}{A_{\cris}}

\newcommand{\tor}{\on{tor}}

\newcommand{\DFL}{D^{\on{FL}}}
\newcommand{\LFL}{L^{\on{FL}}}

\newcommand{\WFLtau}{\mathcal{W}_{\on{FL},\tau}}
\newcommand{\WFLsigma}{\mathcal{W}_{\on{FL},\sigma \tau}}
\newcommand{\DK}{\mathcal{D}_K}

\newcommand{\blockmatrix}{\ttm}

\newcommand{\gr}{\on{gr}}

\newcommand{\nr}{{\on{nr}}}

\newcommand{\Kt}{K^{\on{t}}}

\begin{document}
\title{Producing Geometric Deformations of Orthogonal and Symplectic Galois Representations}
\author{Jeremy Booher}
\date{\today}
\address{Department of Mathematics\\
  University of Arizona\\
  Tucson, Arizona 85721}
\email{jeremybooher@math.arizona.edu}

\begin{abstract}
 For a representation of the absolute Galois group of the rationals over a finite field of characteristic $p$, we study the existence of a lift to characteristic zero that is geometric in the sense of the Fontaine-Mazur conjecture.  For two-dimensional representations, Ramakrishna proved that under technical assumptions odd representations admit geometric lifts.  We generalize this to higher dimensional orthogonal and symplectic representations.  A key step is generalizing and studying a local deformation condition at $p$ arising from Fontaine-Laffaille theory.
\end{abstract}


\maketitle

\tableofcontents

\section{Introduction}

Before the proof by Khare and Winterberger \cite{kw1} \cite{kw2} that irreducible odd representations
\[
 \rhobar: \Gal(\QQbar/\QQ) \to \GL_2(\FFbar_p)
\]
are modular, the lifting result of \cite{ram02} 
together with the Fontaine-Mazur conjecture provided evidence for Serre's conjecture.  Ramakrishna's result shows that under technical hypotheses all odd residual representations admit lifts to characteristic zero that are geometric in the sense of the Fontaine-Mazur conjecture.  Assuming that conjecture, the resulting lifts would be modular as predicted by Serre's conjecture.  
Generalizations of Serre's conjecture to groups other than $\GL_2$ have been proposed, most recently by Gee, Herzig, and Savitt \cite{ghs}, which naturally leads to the problem of producing geometric lifts of Galois representations for groups other than $\GL_2$.

Let $K$ be a finite extension of $\QQ$ with absolute Galois group $\Gamma_K$.  Suppose $k$ is a finite field of characteristic $p$, $\O$ the ring of integers in a $p$-adic field with residue field $k$, and $G$ is a reductive group defined over $\O$.  For a continuous representation $\rhobar: \Gamma_K \to G(k)$,  in light of these conjectures it is important to study when there exists a continuous representation
 $ \rho : \Gamma_K \to G(\O)$ lifting $\rhobar$ that is geometric (using an inclusion of $G$ into $\GL_N$ to define being geometric).
 
 When $G = \GSp_m$ or $G=\GO_m$, we produce geometric lifts in favorable conditions.  The exact hypotheses needed are somewhat complicated.  We will state a simple version now, and defer a more detailed statement to Theorem~\ref{thm:liftingapplication}.  It is essential that $\rhobar$ is odd (as discussed in Remark~\ref{rmk:oddness}, forcing $K$ to be totally real) and that $\rhobar$ restricted to the decomposition group at $p$ ``looks like the reduction of a crystalline representation with \emph{distinct} Hodge-Tate weights''.  More precisely, we assume $p$ is unramified in $K$ and that at places $v$ above $p$, the representation $\rhobar|_{\Gamma_{K_v}}$ is torsion crystalline with Hodge-Tate weights in an interval of length $\frac{p-2}{2}$, so it is Fontaine-Laffaille.  It is crucial that for each $\ZZ_p$-embedding of $\O_{K_v}$ in $\O_{\overline{K_v}}$, the Fontaine-Laffaille weights for $\rhobar|_{\Gamma_{K_v}}$ with respect to that embedding are pairwise distinct (these notions will be reviewed in \S\ref{sec:fontainelaffaille}).
 
 For Ramakrishna's method to apply, it is also essential that the image of $\rhobar$ is ``large'': here we use that $G'(k) \subset \rhobar(\Gamma_K)$ where $G'$ is the derived group.  Ramakrishna's method requires certain technical conditions which follow from this assumption on the image provided that $p > \max(17,2(m-1))$ (this restriction on $p$ is not optimized: see Remark~\ref{rmk:optimalp}).  
Let $\mu : G \to \Gm$ be the similitude character, and define $\nubar = \mu \circ \rhobar : \Gamma_K \to k^\times$.  Suppose there is a lift $\nu : \Gamma_K \to W(k)^\times$ that is Fontaine-Laffaille at all places above $p$.  

\begin{thm} \label{thm:maintheorem}
Let $G = \GSp_m$ or $G=\GO_m$ and let $\rhobar: \Gamma_K \to G(k)$ be an odd representation (which forces $K$ to be totally real and that $m \not \equiv 2 \pmod{4}$ when $G = \GO_m$).  Suppose that $p$ is unramified in $K$ and that at places $v$ above $p$, the representation $\rhobar|_{\Gamma_{K_v}}$ is Fontaine-Laffaille with pairwise distinct weights with respect to each $\ZZ_p$-embedding of $\O_{K_v}$ in $\O_{\overline{K_v}}$.  Furthermore, suppose that $G'(k) \subset \rhobar(\Gamma_K)$ and that $p >  \max(17,2(m-1))$.   Fix a lift $\nu : \Gamma_K \to W(k)^\times$ of $\nubar$ that is Fontaine-Laffaille at all places above $p$.  Then there exists a geometric lift $\rho : \Gamma_K \to G(\O)$ of $\rhobar$ where $\O$ is the ring of integers in a finite extension of $\QQ_p$ with residue field containing $k$ such that $\mu \circ \rho = \nu$.  More precisely, $\rho$ is ramified at finitely many places of $K$, and for every place $v$ of $K$ above $p$ the representation $\rho|_{\Gamma_{K_v}}$ is Fontaine-Laffaille and hence crystalline.  
\end{thm}
 
This provides evidence for generalizations of Serre's conjecture. 
In contrast, when $G = \GL_n$ with $n>2$, the representation $\rhobar$ cannot be odd, and the method does not apply.  In such cases, there is no expectation that such lifts exist.
 
To produce lifts, we use a generalization of Ramakrishna's method also used in \cite{patrikis15}.  
It works by establishing a local-to-global result for lifting Galois representations subject to local constraints (Proposition~\ref{prop:localglobal}).  Let $\rho$ be a lift of $\rhobar$ to $\O/\m^n$ where $\m$ is the maximal ideal of $\O$.  Provided a cohomological obstruction vanishes, it is possible to lift $\rho$ to $\O/\m^{n+1}$ subject to local constraints if (and only if) it is possible to lift $\rho|_{\Gamma_v}$ to $\O/\m^{n+1}$ for all $v$ in a fixed set of places of $K$ containing the places above $p$ and the places where $\rhobar$ is ramified.  Allowing controlled ramification at additional primes kills this obstruction for odd representations.
 
 For this to work, we must pick local deformation conditions above $p$ and at places where $\rhobar$ is ramified which are liftable and have large enough tangent space.  At a prime $\ell \neq p$ where $\rhobar$ is ramified, we use a generalization of the minimally ramified deformation condition defined for $\GL_n$ in \cite[\S 2.4.4]{cht08}.  The correct generalization is not obvious; we define and study this deformation condition in \cite{boohermr}.  At places above $p$, we define a \emph{Fontaine-Laffaille deformation condition} in \S\ref{sec:fldeformation} by using deformations arising from Fontaine-Laffaille modules that carry extra data corresponding to a symmetric or alternating pairing.  

In the remainder of the introduction, we discuss some additional background and give a more detailed overview of the proof.

\subsection{Serre's Conjecture and Geometric Lifts}

We are interested in generalizations of Ramakrishna's lifting result to split reductive groups beyond $\GL_2$, in particular symplectic and orthogonal groups.  
Generalizations of Serre's conjecture have been proposed in this setting, and most of the effort has been to find the correct generalization of the oddness condition and the weight (see for example the discussion in \cite{ghs}, especially \S2.1).  The general flavor of these generalizations is that an odd irreducible Galois representation will be automorphic in the sense that it appears in the cohomology of an $\FFbar_p$-local system on a Shimura variety.  For a general split reductive group, there is no expectation that such representations will lift geometrically to characteristic zero.  For example, as discussed in \cite[\S1]{cht08} the classical Taylor-Wiles method would work only if
\begin{equation} \label{eq:taylorwiles}
 [K:\QQ] \left( \dim G - \dim B \right) = \sum_{v | \infty} \dim H^0(\Gal(\overline{K}_v/K_v) , \adzerorho)
\end{equation}
where $B$ is a Borel subgroup of $G$ and $\adzerorho$ is the adjoint representation of $\Gamma_K$ on the Lie algebra of the derived group of $G$.  Under such a ``numerical coincidence,''  that method gives automorphy lifting theorems and we expect geometric lifts.
 This coincidence cannot hold for $\GL_n$ when $n >2$, but can hold for $G = \GSp_{2n}$ and $G = \GO_m$ when $m \not \equiv 2 \pmod{4}$, and for the group $\mathcal{G}_n$ related to $\GL_n$ considered in \cite{cht08}.  This coincidence is also essential to generalizing Ramakrishna's method.  There are automorphy lifting theorems beyond this setting, but then we don't expect the global deformation ring for $\rhobar$ with bounded ramification and fixed Hodge-Tate weights to always have $\overline{\QQ}_p$-points, and so we cannot produce geometric lifts.

\begin{remark} \label{rmk:oddness}
Following \cite{gross}, we say that $\rhobar : \Gamma_K \to G(k)$ is \emph{odd} if for each archimedean place $v$ and complex conjugation $c_v \in \Gamma_v$ (well-defined up to conjugacy), $\ad(\rhobar(c_v))$ is a split Cartan involution for $\g' := \Lie G^{\on{ad}}$.  Recall that for any involution $\tau$ of $\g'$,
\[
 \dim \left(\mathfrak{\g'}\right)^\tau \geq \dim G - \dim B.
\]
A \emph{split Cartan involution} is an involution for which this is an equality.  If $K$ is totally real and $\rhobar$ is odd, \eqref{eq:taylorwiles} holds.  There are odd representations for symplectic and orthogonal groups, but no odd representations for $\GL_n$ when $n>2$ (for more details, see \cite[\S4.5]{patrikis15}).  These are cases in which we expect geometric lifts, and where Ramakrishna's method generalizes.

There is a less restrictive notion of oddness introduced in \cite[\S6]{bv13}, and the automorphy lifting theorems in \cite{cg12} apply beyond the regime where \eqref{eq:taylorwiles} holds.  
\end{remark}

Ramakrishna developed his lifting technique when $K = \QQ$ and $G = \GL_2$ in \cite{ram99} and \cite{ram02}, and produced geometric lifts.  There have been various reformulations and generalizations that our results build on.  In particular, the formalism developed in \cite{taylor03} (still in the case of $\GL_2$) suggested that it should be possible to generalize the technique to algebraic groups beyonds $\GL_2$.  Attempts were made in \cite{hamblen08} and \cite{mano09} to generalize the technique to $\GL_n$, but ran into the obstruction that there were no odd representations for $n>2$.  The results in \cite{hamblen08} simply assume the existence of liftable local deformation conditions which probably do not exist, but do provide a nice model for generalizing Ramakrishna's method.  In contrast, \cite{mano09} constructs local deformation conditions but does not aim to produce geometric lifts.

For groups beyond $\GL_n$, \cite{cht08} gave a lifting result for a group $\mathcal{G}_n$ which admits odd representations.  By restricting which primes ramify, they can reduce to studying local deformation valued in $\GL_n$.  At primes above $p$, \cite{cht08} studied a deformation condition based on Fontaine-Laffaille theory which is generalized in \S\ref{sec:fldeformation}.  The idea of doing so goes back to \cite{ram93}.  (They also discussed a deformation condition based on the notion of ordinary representations which is not used in their lifting result.)  At primes not above $p$ but where $\rhobar$ is ramified, they defined a \emph{minimally ramified} deformation condition, which is generalized in \cite{boohermr}.  

Building on this, Patrikis' unpublished undergraduate thesis \cite{patrikis06} explored Ramakrishna's method for symplectic groups.  In particular, it generalized Ramakrishna's method to the group $\GSp_n$, and generalized the Fontaine-Laffaille deformation condition at $p$.  It did not 
generalize the minimally ramified deformation condition, so can only be applied to residual representations $\Gamma_\QQ \to \GSp_{n}(k)$ which are unramified away from $p$, a stringent condition.  Our results at $p$ in \S\ref{sec:fldeformation} are a generalization of Patrikis' study of the Fontaine-Laffaille deformation condition.

More recently, Patrikis used Ramakrishna's method to produce geometric representations with exceptional monodromy \cite{patrikis15}.  This involves generalizing Ramakrishna's method to any connected reductive group $G$ and then modifying the technique to deform a representation valued in the principal $\SL_2 \subset G$ (coming from a modular form) to produce a geometric lift with Zariski-dense image.  The generalization of Ramakrishna's method to apply to reductive groups is independently carried out in the author's thesis with only minor technical differences, so in \S\ref{sec:ramakrishnamethod} we refer the reader to \cite{patrikis15} for proofs with a few comments about how to deal with a disconnected group like $\GO_m$.  
The minimally ramified deformation condition of \cite{booher16} is not needed in \cite{patrikis15} as the goal there is just to produce examples of geometric representations with exceptional monodromy.

\begin{remark}
There is also a completely different technique to produce lifts based on automorphy lifting theorems.  For example, Khare and Winterberger use it in their proof of Serre's conjecture: see \cite[\S4]{kw2} especially the proof of Corollary 4.7.  
The key ingredients are the computations of the dimension of components of the generic fiber of local deformation rings, the fact that these local deformation rings have non-trivial generic fiber, and 
the fact that a suitable global deformation ring is a finite $\O$-algebra.  

The finiteness of the global deformation ring can be established 
by relating the Galois deformation ring to a Hecke algebra using a suitable automorphy lifting theorem or potential automorphy theorem. 
If the local deformation rings have non-trivial generic fiber, information about the dimension lets one show that for an odd representation the dimension of the global deformation ring is at least one.  This implies the existence of geometric lifts.
This approach avoids a detailed analysis of the local deformation rings, and also allows more control of the local properties of the lift.  
In particular, it is not necessary to allow the lift to ramify at places beyond the places where $\rhobar$ is ramified.  However, the local calculations with Fontaine-Laffaille theory in this paper, and the calculations in \cite{boohermr} are still relevant since they provide a way to check that the generic fiber of the local deformation rings are non-empty.
\end{remark}

\subsection{Overview of the Proof}

The argument to produce geometric deformations has two main components: a global argument involving Galois cohomology that uses local deformation conditions as black boxes, and the construction and analysis of local deformation.

The first part of the argument, with only minor technical variation, has also been carried out in \cite{patrikis15}.
Fix a prime $p$ and finite field $k$ of characteristic $p$.  Let $S$ be a finite set of places of a number field $K$ containing the places above $p$ and the archimedean places, and define $\Gamma_S$ to be the Galois group of the maximal extension of $K$ unramified outside of $S$.  Consider a continuous representation $\rhobar : \Gamma_S \to G(k)$ where $G$ 
is a smooth affine group scheme over the ring of integers $\O$ in a $p$-adic field such that the identity components of the fibers are reductive.  We are mainly interested in the case that $G=\GSp_m$ or $G = \GO_m$; the latter may have disconnected fibers. (In the relative setting, by definition reductive groups have connected fibers, so we must work in slightly greater generality as discussed at the start of \S\ref{sec:alggrps}.)  Assume that $p$ is  very good for $G$ (Definition~\ref{defn:goodprime}).

The hope would be to use deformation theory to produce $\rho_n : \Gamma_S \to G(\O/\m^n)$ such that $\rho_1 = \rhobar$, $\rho_n$ lifts $\rho_{n-1}$ for $n \geq 2$, and such that $\rho_n$ satisfies a deformation condition at places above $p$ for which the inverse limit 
\[
\rho =  \inverselimit \rho_n : \Gamma_S \to G(\O)
\]
restricted to the decomposition group $\Gamma_v$ would be a lattice in a de Rham (or crystalline) representation for places $v$ of $K$ above $p$.  This inverse limit would then be the desired geometric lift of $\rhobar$.  
Only after a careful choice of local deformation conditions and enlarging the set $S$ will this work.  Furthermore, defining these deformation conditions may require making an extension of $k$, which is harmless for our applications and is why we only require that the residue field of $\O$ contains $k$.

Proposition~\ref{prop:localglobal} gives a local-to-global principle for lifting $\rho_{n-1}$ to $\rho_n$ subject to a global deformation condition $\D_S$: provided the dual Selmer group $H^1_{\D_S^\perp}(\Gamma_S,\adzerorho^*)$ vanishes, it is possible to produce global lifts subject to this condition if it is possible to lift each $\rho_{n-1}|_{\Gamma_v}$ subject to the local conditions.
This Galois cohomology group is defined in \eqref{eq:dualselmer}, and encodes information about the local deformation conditions. 
Proposition~\ref{prop:killingdualselmer} gives a way to enlarge $S$ and $\D_S$, allowing ramification subject to Ramakrishna's deformation condition at the new places, that forces $H^1_{\D_S^\perp}(\Gamma_S,\adzerorho^*)$ to be zero.  We review Ramakrishna's deformation condition in \S\ref{sec:ramakrishnacondition}.  The places of $K$ at which we define this condition are found using the Chebotarev density theorem: each additional place where we allow ramification subject to Ramakrishna's deformation condition decreases the dimension of the dual Selmer group.  For such places to exist, we need non-zero classes in certain cohomology groups, whose existence relies on the local deformation conditions satisfying the inequality
\begin{equation} \label{earlytspace}
  \sum_{v \in S} \dim L_v \geq \sum_{v \in S} \dim H^0(\Gamma_v,\adzerorho),
\end{equation}
where $L_v$ is the tangent space of the local deformation condition at $v$.  Furthermore,  $\rhobar$ needs to be a ``big'' representation in the sense of Definition~\ref{def:bigrep} in order to define Ramakrishna's deformation condition.  Being a big representation is a more precise set of technical conditions that are implied for large enough $p$ by the condition $G'(k) \subset \rhobar(\Gamma_K)$ appearing in Theorem~\ref{thm:maintheorem}.  

For \eqref{earlytspace} to hold, it is crucial that $\rhobar$ be an odd representation.  The minimally ramified deformation condition at places $v$ where $\rhobar$ is ramified studied in  \cite{boohermr} is liftable and satisfies                                                                                                                                                                                                                                                        
$\dim H^0(\Gamma_v,\adzerorho)  = \dim L_v$.  We will define a Fontaine-Laffaille deformation condition at places above $p$.  Using it, \eqref{earlytspace} becomes
\[
[K:\QQ]  (\dim G - \dim B)   \geq \sum_{v | \infty} h^0(\Gamma_v,\adzerorho) 
\]
where $B$ is a Borel subgroup of $G$; this can only be satisfied if $K$ is totally real and $\rhobar$ is odd. 

The other key part of the argument is to generalize the Fontaine-Laffaille deformation condition.  
Let $K$ be a finite unramified extension of $\QQ_p$, and let $\O$ be the ring of integers of a $p$-adic field $L$ with residue field $k$ such that $L$ splits $K$ over $\QQ_p$.  (The latter is always possible after extending $k$.)  
Fontaine-Laffaille theory, introduced in \cite{fl82}, provides a way to describe torsion-crystalline representations with Hodge-Tate weights in an interval of length $p-2$ in terms of semi-linear algebra as $p$ is unramified in $K$.  In particular, it provides an exact, fully faithful functor $T_\cris$ from the category of filtered Dieudonn\'{e} modules to the category of $\O[\Gamma_K]$-modules with continuous action, and describes the image (Fact~\ref{fact:tcris}).
In \cite[\S2.4.1]{cht08}, it is used to define a deformation condition for $\GL_n$, where the allowable deformations of $\rhobar$ are exactly the deformations of the corresponding Fontaine-Laffaille module.  This requires the technical assumption that the representation $\rhobar$ is torsion-crystalline with Hodge-Tate weights in an interval of length $p-2$.  The deformation condition is liftable of the desired dimension provided that the Fontaine-Laffaille weights of $\rhobar$ under each embedding of $K$ into $L$ are distinct (see Remark~\ref{remark:embeddingweights}).

We will adapt these ideas to symplectic and orthogonal groups under the assumption that the Fontaine-Laffaille weights lie in an interval of length $\frac{p-2}{2}$.  For symplectic groups and $K = \QQ_p$, this was addressed in Patrikis's undergraduate thesis \cite{patrikis06}: we generalize this, and record proofs as the thesis is not readily available.  The key idea is to introduce a symmetric or alternating pairing into the semi-linear algebra data.
To do so, it is necessary to use (at least implicitly via statements about duality) the fact that the functor $T_\cris$ is compatible with tensor products.  This requires the stronger assumption that the Fontaine-Laffaille weights lie in an interval of length $\frac{p-2}{2}$, which guarantees that the Fontaine-Laffaille weights of the tensor product lie in an interval of length $p-2$.  Furthermore, it is crucial to use the covariant version of the Fontaine-Laffaille functor used in \cite{bk90} instead of the contravariant version studied in \cite{fl82} in order for the compatibility with tensor products to hold.  For more details, see \S\ref{sec:tensorfree}.  Given this, it is then reasonably straightforward to check that $T_\cris$ is compatible with duality and hence to translate the (perfect) alternating or symmetric pairing of Galois representations into a (perfect) symmetric or alternating pairing of Fontaine-Laffaille modules.

For a coefficient ring $R$, define $\DFL_\rhobar(R)$ to be all representations $\rho : \Gamma_K \to G(R)$ lifting $\rhobar$ and lying in the essential image of $T_\cris$.  To study this Fontaine-Laffaille deformation condition, it suffices to study Fontaine-Laffaille modules.  In particular, to show that the deformation condition is liftable (i.e. that it is always possible to lift a deformation satisfying the condition through a square-zero extension), it suffices to show that a Fontaine-Laffaille module with distinct Fontaine-Laffaille weights together with a perfect symmetric or skew-symmetric pairing can always be lifted through a square zero extension.  This is a complicated but tractable problem in semi-linear algebra: Proposition~\ref{prop:flliftable} shows this is always possible.  It is relatively simple to lift the underlying filtered module and the pairing, and requires more care to lift the semi-linear maps $\varphi^i_M : M^i \to M$.  Likewise, to understand the tangent space of the 
deformation condition it suffices to 
study deformations of the Fontaine-Laffaille module corresponding to $\rhobar$ to the dual numbers.  Again, the most involved step is understanding possible lifts of the semi-linear maps after choosing a lift of the filtration and the pairing.

\begin{remark}
The proof that $\DFL_\rhobar$ is liftable and the computation of the dimension of its tangent space both use in an essential way the hypothesis that for each embedding of $K$ into $L$ the Fontaine-Laffaille weights are pairwise distinct.  
\end{remark}

\begin{remark}
An alternative deformation condition to use at primes above $p$ is a deformation condition based on the concept of an ordinary representation.  This is studied for any connected reductive group in \cite[\S4.1]{patrikis15}.  It is suitable for use in Ramakrishna's method, and can give lifting results for a different class of torsion-crystalline representations.  
\end{remark}

\subsection{Acknowledgments}
This work formed part of my thesis~\cite{booher16}, and I am grateful for the generosity and support of my advisor Brian Conrad, and for his extensive and helpful comments on drafts of my thesis.  I thank Brandon Levin for bringing \cite{patrikis06} to my attention, and Stefan Patrikis for his encouragement and conversations.  I thank the referees for many helpful suggestions and careful reading.

\section{Deformations of Galois Representations}

\subsection{Algebraic Groups and Very Good Primes} \label{sec:alggrps}

Let $\O$ be a discrete valuation ring with residue field $k$ of characteristic $p$.  Let $G$ be smooth separated group scheme over $\O$ such that the identity components of the fibers are reductive.\footnote{For results about reductive group schemes, we refer to \cite{conrad14} which gives a self-contained development, using more recent methods, of results from \cite{sga3}.}  Then $G^\circ$ is a reductive $\O$-subgroup scheme of $G$ and $G/G^\circ$ is a separated \'{e}tale $\O$-group scheme of finite presentation~\cite[Proposition 3.1.3 and Theorem 5.3.5]{conrad14}.   Furthermore, by a result of Raynaud $G$ is affine as it is a flat, separated, and of finite type with affine generic fiber over the discrete valuation ring $\O$ \cite[Proposition 3.1]{py06}.  Call such $G$ almost-reductive group schemes over $\O$.  We say $G$ is split if $G^\circ$ is split.

\begin{remark}
A reductive group scheme has connected fibers by definition: see \cite[Definition 3.1.1]{conrad14}, going back to \cite[XIX, 2.7]{sga3}).  Connectedness is important as in general the component group may jump across fibers.  We wish to be able to work with $\GO_m$ which may have two connected components, so we work in this generality.
\end{remark}

Let $\Phi$ be a reduced and irreducible root system, and $P$ the weight lattice for $\Phi$.  We recall the notion of a very good prime.

\begin{defn} \label{defn:goodprime}
The prime $p$ is \emph{good} for $\Phi$ provided that $\ZZ \Phi / \ZZ \Phi'$ is $p$-torsion free for all subsets $\Phi'\subset \Phi$.  A good prime is \emph{very good} provided that $P / \ZZ \Phi'$ is $p$-torsion free for all subsets $\Phi' \subset \Phi$.  A prime is \emph{bad} if it is not good.
\end{defn}

Likewise, we say a prime $p$ is \emph{good} (or \emph{very good}) for a general reduced root system if it is good (or very good) for each irreducible component.  A prime $p$ is \emph{good} (or \emph{very good}) for $G$ provided it is good (or very good) for the root system of $G^\circ_{\kbar}$.
For example, if $G = \GSp_{2n}$ or $G = \GO_m$ every prime except $2$ is very good.  The prime $p$ being very good for a split almost-reductive group scheme $G$ for example implies that:
\begin{itemize}
 \item the center of $\Lie G_k$ is the Lie algebra of $Z_{G_k}$, and $\Lie G_k$ is a direct sum of $\Lie G'_k$ and $\Lie Z_{G_k}$, where $G'$ is the derived group of $G^\circ$ and $Z_{G_k}$ is the center of $G^\circ_k$;
 \item  $Z_{G'_k}$ and $\pi_1(G^\circ_k)$ have order prime to $p$.
\end{itemize}
These facts are well-known.

\subsection{Deformation Functors} \label{defsection}  Next we briefly summarize some facts about the deformation theory for Galois representations.  For a more detailed introduction, see \cite[\S2.1]{boohermr} or the comprehensive reference \cite{mazur95}, with the extension to algebraic groups beyond $\GL_n$ in \cite{tilouine96}.

Let $\Gamma$ be a pro-finite group satisfying the following finiteness property: for every open subgroup $\Gamma_0 \subset \Gamma$, there are only finitely many continuous homomorphisms from $\Gamma_0$ to $\ZZ/ p \ZZ$.  This is true for the absolute Galois group of a local field and for the Galois group of the maximal extension of a number field unramified outside a finite set of places.  Let $\rhobar: \Gamma \to G(k)$ be a continuous homomorphism.  

A coefficient $\O$-algebra $R$ is a complete local Noetherian $\O$-algebra with residue field $k$: a \emph{lift} of $\rhobar$ to $R$ is a continuous homomorphism $\rho : \Gamma \to G(R)$ that reduces to $\rhobar$.  A \emph{deformation} is an equivalence class of lifts under conjugation by a $ g \in G(R)$ which reduces to the identity.  The universal lifting ring (respectively universal deformation ring) is denoted $R_{\rhobar}^\square$ (respectively $R_\rhobar$ provided it exists).
 While we usually care about deformations, it is technically easier to work with lifts as $R_\rhobar^\square$ always exists.

The deformation theory of Galois representations is controlled by Galois cohomology.  In particular, $H^2(\Gamma, \adrho)$ controls liftability and the tangent space of the deformation functor is isomorphic to $H^1(\Gamma,\adrho)$, where $\adrho$ denotes the representation of $\Gamma$ on $\g_k = \Lie G_k$ via the adjoint representation.  Usually, we care about cohomology valued in the subspace $\adzerorho$ consisting of the Lie algebra of the derived group of $G^\circ$.  
 Since $p$ is very good, $\g_k = \g'_k \oplus \zfrak_\g$ where $\zfrak_\g$ is the Lie algebra of $Z_G$ and the natural map $H^i(\Gamma,\adzerorho) \to H^i(\Gamma,\adrho)$ is injective for all $i$; we often use this without comment.

Recall that a \emph{lifting condition} is a sub-functor $\D^\square$ of the functor of lifts $\D^\square_\rhobar$ (from coefficient $\O$-algebras to sets) such that:
\begin{enumerate}
 \item  For any coefficient ring $A$, $\D^\square(A)$ is closed under strict equivalence.
  \item  \label{defcon2} Given a Cartesian diagram in $\C_\O$
 \[
  \xymatrix{
  A_1 \times_{A_0} A_2 \ar[r]^-{\pi_2} \ar[d]^-{\pi_1} & A_2\ar[d]\\
  A_1 \ar[r] & A_0
  }
 \]
 and $\rho \in D^\square_{\rhobar}(A_1 \times_{A_0} A_2)$, we have $\rho \in \D^\square(A_1 \times_{A_0} A_2)$ if and only if $\D^\square(\pi_1) \circ \rho \in \D^\square(A_1)$ and $\D^\square(\pi_2) \circ \rho \in \D^\square(A_2)$.  
\end{enumerate}
As it is closed under strict equivalence, we naturally obtain a \emph{deformation condition}, a sub-functor $\D$ of the functor of deformations.
According to Schlessinger's criterion \cite[Theorem 2.11]{schlessinger68}, this definition is equivalent to the functor $\D^\square$ being pro-representable.  Likewise, a deformation condition $\D$ is pro-representable provided that $D_{\rhobar}$ is.

For a deformation condition $\D$, we denote its tangent space by $H^1_{\D}(\Gamma,\adrho)$; it is a $k$-subspace of  $H^1(\Gamma,\adrho)$.  The set of deformations through a small surjection subject to $\D$ is a $H^1_{\D}(\Gamma,\adrho)$-torsor.  The torsor structure is compatible with the action of $H^1(\Gamma,\adrho)$ on the space of all deformations.

Recall that a deformation condition $\D$ is \emph{locally liftable} (over $\O$) if for all small surjections $f : A_1 \to A_0$ of coefficient $\O$-algebras the natural map
\[
 \D(f) : \D(A_1) \to \D(A_0)
\]
is surjective.
A geometric way to check local liftability is to show that the corresponding lifting ring or deformation ring (when it exists) is formally smooth over $\O$.

\subsection{Global Deformations}

We now review global deformation conditions.  Let $K$ be a number field, $S$ a finite set of places of $K$ that contains all the places of $K$ at which $\rhobar$ is ramified and all archimedean places.  Let $\Gamma_S$ be the Galois group of the maximal extension of $K$ unramified outside of $S$ and $\Gamma_K$ be the absolute Galois group of $K$. 

\begin{defn}
 A \emph{global deformation condition} $\D_S$ for $\rhobar : \Gamma_S \to G(k)$ is a collection of local deformation conditions $\{\D_v\}_{v \in S}$ for $\rhobar|_{\Gamma_v}$.  We say it is \emph{locally liftable} (over $\O$) if each $\D_v$ is locally liftable (over $\O$).  A \emph{global deformation of $\rhobar : \Gamma_S \to G(k)$ subject to $\D_S$} is a deformation $\rho : \Gamma_S \to G(A)$ such that $\rho|_{\Gamma_v} \in \D_v(A)$ for all $v \in S$.
\end{defn}

For $v \in S$, let $L_v$ denote the tangent space of the local deformation condition $\D_v$.
A global deformation condition gives a generalized Selmer group.  We will be mainly interested in the \emph{dual Selmer group}
\begin{equation} \label{eq:dualselmer}
 H^1_{\D_S^\perp} ( \Gamma_S, \adrho^*) = \{ x \in H^1(\Gamma_S,\adrho^*) : \res_v(x) \in L_v^\perp \, \text{for all}\, v \in S\}.
\end{equation}

For Ramakrishna's method to work, it is crucial that the local tangent spaces be large enough relative to the local invariants.  We say that a global deformation condition satisfies the \emph{tangent space inequality} if 
\begin{equation} \label{eq:tspaceinequality}
 \sum_{v \in S} \dim L_v \geq \sum_{v \in S} \dim H^0(\Gamma_v,\adzerorho).
\end{equation}

 Let $\D_S = \{ \D_v\}$ be a global deformation condition, and $G'$ be the derived group of $G^\circ$ with quotient $\mu : G \to G/G'$.  We assume that the deformation condition includes the condition of fixing a lift $\nu : \Gamma_K \to (G/G')(\O)$ of the character $\mu \circ \rhobar : \Gamma_K \to (G/G')(k)$.  This means that all of the local deformation conditions have tangent spaces lying in $H^1(\Gamma_v,\adzerorho)$, and the obstruction cocycles automatically land in $H^2(\Gamma_v,\adzerorho)$ (see Example~2.4 and Example 2.6 of \cite{boohermr}), with similar statements for global deformation conditions.  In favorable circumstances, we can use the following local-to-global principle to produce lifts.
 
\begin{prop} \label{prop:localglobal}
Let $A_1 \to A_0$ be a small extension of coefficient $\O$-algebras with kernel $I$, and consider a lift $\rho_0 : \Gamma_S \to G(A_0)$ of $\rhobar$ subject to $\D_S$.  Provided $H^1_{\D_S^\perp}(\Gamma_S,\adzerorho^*)=0$, lifting $\rho_0$ to $A_1$  subject to $\D_S$ is equivalent to lifting $\rho_0 |_{\Gamma_v}$ to $A_1$ subject to $\D_v$ for all $v \in S$.  
\end{prop}

\begin{proof}
When $G = \GL_2$, this is \cite[Lemma 1.1]{taylor03}.  The statement and proof of that Lemma work without change in our setting.  

\end{proof}


\section{Generalizing Ramakrishna's Method} \label{sec:ramakrishnamethod}

The key to generalizing Ramakrishna's method is the ability to choose local conditions so that Proposition~\ref{prop:localglobal} will apply.  This generalization is carried for split reductive group schemes with connected fibers in the author's thesis and in \cite{patrikis15} with only minor technical differences between them, such as the fact that \cite{patrikis15} also treats L-groups.
 Here we refer to \cite{patrikis15} for proofs and only point out the modifications necessary to deal with split almost-reductive groups like $\GO_m$.  So let $\O$ be the ring of integers in a $p$-adic field with residue field $k$, and let $q = \# k$.  Consider a split almost-reductive group scheme $G$ over $\O$ with Lie algebra $\g$.  Let $K$ be a number field and denote the $p$-adic cyclotomic character by $\chi : \Gamma_K \to \ZZ_p^\times$, with reduction $\overline{\chi} : \Gamma_K \to \FF_p^\times$.  Fix a split maximal torus $T \subset G^\circ$.

\subsection{Ramakrishna's Deformation Condition} \label{sec:ramakrishnacondition}

We start by assuming:
\begin{enumerate}[label=(A\arabic*)]
 \item \label{a0} there is $\gamma \in \Gamma_K$ such that $\rhobar(\gamma)\in G^\circ(k)$ is regular semisimple, and $Z_{G_k}(\rhobar(\gamma))^\circ = T_k$;
 \item \label{a2} there is a root $\alpha \in \Phi(G,T)$ such that $\alpha(\rhobar(\gamma)) = \overline{\chi}(\gamma)$;  
 \item \label{a1} there is a place $v$ of $K$ lying over a rational prime $\ell$ such that $\rhobar$ is unramified at $v$, $\rhobar(\Gamma_v) \subset G^\circ(k)$, and $\rhobar(\Frob_v)$ is regular semisimple element.  The identity component of $Z_{G_k}(\rhobar(\Frob_v))$ is $T_k$, and $\alpha(\rhobar(\Frob_v)) = \chibar(\Frob_v) = \chibar(\gamma) \neq 1$.
\end{enumerate}
Under these assumptions, we can define Ramakrishna's deformation condition $\Dram_v$ for $\rho_v: \Gamma_v \to G^\circ(k)$ as in \cite[\S4.2]{patrikis15}. 
We form the root group $U_\alpha \subset G^\circ$ associated to $\alpha$.

\begin{defn}
For a coefficient $\O$-algebra $A$, consider a lift $\rho : \Gamma_v^{\textrm{t}} \to G^\circ(A)$.  The lift $\rho$ satisfies \emph{Ramakrishna's condition relative to $T$} provided that $\rho(\Frob_v) \in T(A)$, $\alpha(\rho(\Frob_v)) = \chi(\Frob_v)$, and $\rho(\Gal(\Kt_v/K_v^{\nr})) \subset U_\alpha(A) \subset G^\circ(A)$.

Define \emph{Ramakrishna's deformation condition} $\Dram_v(A)$ to be lifts which are $\Ghat(A)$-conjugate to one which satisfies Ramakrishna's condition relative to $T$. 
\end{defn}

Letting $S$ be the quotient of $G^\circ$ by its derived group with quotient map $\mu$, we can also study lifts $\rho : \Gamma_{K_v} \to G^\circ(A)$ such that $\mu \circ \rho$ is a fixed unramified lift $\nu$ of $\mu \circ \rhobar$.  As the condition $\mu \circ \rho = \nu$ cuts out a closed subscheme of the universal lifting ring for $\Dram_v$, this is a deformation condition we will denote by $\D^{\on{ram}, \nu}_v$.  Lemmas 4.10 and 4.11 of \cite{patrikis15} show:

\begin{fact}
 The deformation conditions $\Dram_v$ and $\D^{\on{ram}, \nu}_v$ are liftable.  The dimension of their tangent spaces are $\dim H^0(\Gamma_v,\adrho)$ and $\dim H^0(\Gamma_v,\adzerorho)$ respectively.
\end{fact}

\begin{remark}
In order to apply the results of \cite[\S4.2]{patrikis15} (or the analogous results in \cite[\S2.4]{booher16}) at the place $v$, it is important to have that $\rhobar(\Gamma_v) \subset G^\circ(k)$.  Similarly, when analyzing disconnected $L$-groups \cite[\S9.2]{patrikis15} reduces to situations where the Galois-representation on the (constant) component group scheme is trivial in order to use this situation.
\end{remark}

\subsection{Big Representations}
Let $K(\adzerorho)$ and $K(\adzerorho^*)$ denote the fixed field of the kernel of the actions of $\Gamma_K$ on $\adzerorho$ and $\adzerorho^*$ respectively, and let $F$ be the compositum.  Let $\D_S$ be a global deformation condition satisfying the tangent space inequality \eqref{eq:tspaceinequality}.
The natural class of representations $\rhobar : \Gamma_K \to G(k)$ to which Ramakrishna's method will apply are those which satisfy the following conditions:

\begin{defn} \label{def:bigrep}
A representation $\rhobar : \Gamma_K \to G(k)$ is $\emph{big}$ relative to $\D_S$ provided that
 \begin{enumerate}[(i)]
  \item \label{bigrep1} we have $H^0(\Gamma_K,\adzerorho) = H^0(\Gamma_K,\adzerorho^*) =0$; 
  \item \label{bigrep2} we have $H^1(\Gal(K(\adzerorho)/K),\adzerorho)  = 0$ and $H^1(\Gal(K(\adzerorho^*)/K) ,\adzerorho^*) = 0$;
  \item \label{bigrep3}for any non-zero $\psi \in H^1_{\D_S}(\Gamma_S,\adzerorho)$ and $\phi \in H^1_{\D_S^\perp}(\Gamma_S,\adzerorho^*)$, the fields $F_\psi$ and $F_\phi$ are linearly disjoint over $F$, where $F_\psi$ (respectively $F_\phi$) is the fixed field of the kernel of the homomorphism obtained by restricting $\psi$ (respectively $\phi$) to $\Gamma_F$;
  \item  \label{bigrep4}for any non-zero $\psi \in H^1_{\D_S}(\Gamma_S,\adzerorho)$ and $\phi \in H^1_{\D_S^\perp}(\Gamma_S,\adzerorho^*)$, there is an element $\gamma \in \Gamma_K$ such that $\rhobar(\gamma) \in G^\circ(k)$ is regular semisimple with $Z_{G_k}(\rhobar(\gamma))^\circ = T_k$, and for which there is a root $\alpha \in \Phi(G,T)$ satisfying $\alpha(\rhobar(\gamma)) = \overline{\chi}(\gamma) \neq 1$ and 
  (letting $\t_\alpha$ denote the span of the $\alpha$-coroot vector and $\g_{-\alpha}$ denote the $-\alpha$ root space) for which $k[\psi(\Gamma_K)]$ has an element with non-zero $\t_\alpha$-component, and for which $k[\phi(\Gamma_K)]$ has an element with non-zero $\g_{-\alpha}$-component.
 \end{enumerate}
\end{defn}

\begin{remark}
 In \ref{bigrep4}, note that $\alpha(\rhobar(\gamma))$ makes sense because $\rhobar(\gamma) \in T(k)$, as any semisimple element $g \in G^\circ(k)$ satisfies $g \in Z_{G_k}(g)^\circ$.  
\end{remark}

\begin{remark} \label{rmk:big}
Observe that these conditions are insensitive to extension of $k$.
\end{remark}

Let $S$ be a finite set of places of $K$ containing the archimedean places, the places over $p$, and the places where $\rhobar$ is ramified.  

\begin{prop} \label{prop:killingdualselmer}
Let $\D_S$ be a global deformation condition that satisfies the tangent space inequality, and suppose $\rhobar$ is big relative to $\D_S$.  There is a finite set of places $T \supset S$ such that the deformation condition $\D_T$ obtained by extending $\D_S$ allowing deformations according to $\Dram_v$ for $v \in T \backslash S$ satisfies
\[
 H^1_{\D_T^\perp}(\Gamma_K,\adzerorho^*) = 0.
\]
\end{prop}

\begin{proof}
The connected case is \cite[Proposition 5.2]{patrikis15}.  There, we find places $v$ of $K$ satisfying the hypotheses necessary to define Ramakrishna's deformation condition using the Chebotarev density theorem on the extension $F_\psi F_\phi K(\rhobar)/ K$, where $K(\rhobar)$ is the fixed field of the kernel of $\rhobar$, using as input the element $\gamma$ in the definition of bigness.  A version for $L$-groups is \cite[Proposition 10.2]{patrikis15}, and we can likewise adapt the arguments to our situation.  The difference from the connected case is that we have the additional requirement that $\rhobar(\Gamma_v) \subset G^\circ(k)$, or equivalently $\rhobar(\Frob_v) \in G^\circ(k)$ as $\rhobar$ is unramified at $v$.
 
Let $K'$ denote the fixed field of the kernel of the composition of $\rhobar$ with the map to the component group of $G_k$.  We apply the Chebotarev density theorem to the extension $F_\psi F_\phi K(\rhobar) / K'$, using that $\rhobar(\gamma) \in G^\circ(k)$, obtaining a place $v'$ with $\rhobar(\Frob_{v'}) \in G^\circ(k)$ as well as the original conditions.  As the primes of $K'$ which are split over $K$ have density $1$, we may freely add the condition that the place $v'$ of $K'$ is split over the place $v$ of $K$.  As $K'_{v'} = K_v$, we conclude that $\rhobar(\Frob_v) \subset G^\circ(k)$.  The original argument then shows that adding Ramakrishna's deformation condition at $v$ to the global deformation condition decreases the size of the dual Selmer group.
\end{proof}

There is an easy case in which we can check that $\rhobar$ is big relative to a global deformation $\D_S$ satisfying the tangent space inequality.  Let $G'$ be the derived group of $G^\circ$, and $h$ the Coxeter number of $G'$.

\begin{prop} \label{prop:bigimage}
Suppose that $K \cap \QQ(\mu_p) = \QQ$, that $p$ is relatively prime to the order of the component group of $G_k$, and that the root system of $G^\circ$ is irreducible and of rank greater than $1$.  If $G'(k) \subset \rhobar(\Gamma_S)$, and $p-1$ is greater than the maximum of $8 \# Z_{G'}$ and 
\[
 \begin{cases}
  (h-1) \# Z_{G'} & \text{ if } \# Z_{G'} \text{ is even}\\
  (2h-2) \# Z_{G'} & \text{ if } \# Z_{G'} \text{ is odd}
 \end{cases}
\]
then $\rhobar$ is big relative to $\D_S$.
\end{prop}

\begin{proof}
 This is part of the proof of \cite[Theorem 6.4]{patrikis15}.  Small modifications are needed to deal with almost-reductive $G$.  In particular, when deducing \ref{bigrep2}, it is necessary to use inflation-restriction to pass from the statement that
 $ H^1(G'(k),\adzerorho)=0$ to the statement that $H^1(\rhobar(\Gamma_S),\adzerorho)=0$ using that the index of $G'(k) \subset G^\circ(k) \subset G(k)$ is prime to $p$.  The arguments for \ref{bigrep3} and\ref{bigrep4} are unchanged: both rely on constructing elements in the image of $\rhobar$ using root data, so the argument can take place inside $G^\circ$.  
\end{proof}

\begin{remark} \label{rmk:optimalp}
The argument is not optimized to produce the weakest restriction on $p$.  The approach works uniformly for any irreducible root system: in any specific case improvements should be possible. 
\end{remark}

\begin{remark}
The formulation in \cite[\S2.3]{booher16} is very similar (only treating the case that $G$ has connected fibers).  Conditions \ref{bigrep1}, \ref{bigrep3}, and \ref{bigrep4} are replaced by the simpler but stronger conditions that $\adzerorho$ is an absolutely irreducible representation of $\Gamma_K$ and the condition that
\begin{enumerate}
 \item[(iii)] \label{bigrepiii} there exists $\gamma \in \Gamma_K$ such that $\rhobar(\gamma)$ is regular semisimple with associated maximal torus $Z_{G_k}(\rhobar(\gamma))^\circ$ equal to the split maximal torus $T_k$, and for which there is a unique root $\alpha \in \Phi(G,T)$ satisfying $\alpha(\rhobar(\gamma)) = \overline{\chi}(\gamma) \neq 1$.  (If $\dim T= 1$, we furthermore require that $\chibar(\gamma)^3 \neq 1$.)
\end{enumerate}
This condition holds in the situation of Proposition~\ref{prop:bigimage}.  The analysis follows analogous lines.
The conditions that the root system of $G^\circ$ is irreducible and that $G'$ is not of rank $1$ are removed by additional bookkeeping and imposing a stronger bound on $p$ when the rank of $G'$ is $1$.
\end{remark}

\subsection{Choosing Deformation Conditions}

Let $G'$ be the derived group of $G^\circ$ and $\mu : G \to G/G'$ be the quotient map.  For a fixed lift $\nu$ of $$\mu \circ \rhobar : \Gamma_K \to (G/G')(k),$$ the heart of the matter is to choose deformations conditions so that we may apply Proposition~\ref{prop:localglobal} and Proposition~\ref{prop:killingdualselmer} to produce a geometric lift of $\rhobar$ with $\mu \circ \rhobar = \nu$.   We need:
\begin{enumerate}
 \item  Locally liftable deformation conditions at finite places away from $p$ where $\rhobar$ is ramified.
 \item Locally liftable deformation conditions at places above $p$ whose characteristic-zero points are lattices in crystalline (or semistable) representations. 
 \item  The tangent space inequality \eqref{eq:tspaceinequality} to hold, which will require $\rhobar$ to be odd.
\end{enumerate}

It is necessary to extend $\O$ and $k$ in order to define some of these deformation conditions: the condition that $\rhobar$ is big is unaffected (Remark~\ref{rmk:big}), so we are free to do so.  We will find such deformation conditions when $G = \GSp_{m}$ with even $m $ or for $G = \GO_m$.  In order to have the necessary oddness assumption on $\rhobar$, in the latter case $m \not \equiv 2 \pmod{4}$.

At the places where $\rhobar$ is ramified, we use the minimally ramified deformation condition studied in \cite{boohermr}.  In particular, \cite[Theorem 1.1]{boohermr} gives:

\begin{fact}
Let $v$ be a place not dividing $p$ at which $\rhobar$ is ramified.  After a finite extension of $k$ (and $\O$), we can define the minimally ramified deformation condition with fixed similitude character $\nu_v$.  It is  liftable, and its tangent space has dimension $\dim H^0(\Gamma_v, \adzerorho)$.
\end{fact}

At the places above $p$, when $G = \GO_m$ or $\GSp_{m}$ after extending $k$ we will construct a \emph{Fontaine-Laffaille deformation condition} using Fontaine-Laffaille theory in \S\ref{sec:fldeformation}.  This requires the assumption that $\nu \otimes \O[\frac{1}{p}]$ is crystalline, $p$ is unramified in $K$, $\rhobar$ is torsion-crystalline with Hodge-Tate weights in an interval of length $\frac{p-2}{2}$, and that the Fontaine-Laffaille weights for each $\ZZ_p$-embedding of $\O_K$ into $\O$ are pairwise distinct.
The deformation condition is liftable, and the dimension of the tangent space will be $h^0(\Gamma_v,\adzerorho) + [K_v:\QQ_p] (\dim G_k - \dim B_k) $, where $B$ is a Borel subgroup of $G$.  This generalizes the results for $\GL_n$ obtained in \cite[\S2.4.2]{cht08}.

\begin{remark}
The restriction that $p$ is unramified in $K$ and that the Hodge-Tate weights of $\rhobar$ are in an interval of length $\frac{p-2}{2}$ is required to use Fontaine-Laffaille theory.  Approaches using different flavors of integral $p$-adic Hodge theory should be able to remove it (for example, the deformation condition based on ordinary representations worked out by Patrikis \cite[\S4.1]{patrikis15} does so for a special class of representations).  However, most previous work on studying deformation rings using integral $p$-adic Hodge theory only gives results about the crystalline deformation ring with $p$ inverted, which does not suffice for our method.
 
The assumption that the Hodge-Tate weights are pairwise distinct is crucial, as otherwise the expected dimensions of the local crystalline deformation rings are too small to use in Ramakrishna's method.
\end{remark}

We also need to specify a deformation condition at the archimedean places $v$: we just require lifts for which $\mu \circ \rho|_{\Gamma_v} =  \nu|_{\Gamma_v}$.  This condition is very simple to arrange, as $\# \Gamma_v \leq 2$.  At a complex place, the dimension of the tangent space is zero and the dimension of the invariants is $\dim_k \adzerorho$.  At a real place, the tangent space is zero when $p>2$ and the invariants are the invariants of complex conjugation on $\adzerorho$.

Now we study the tangent space inequality \eqref{eq:tspaceinequality}.  Let $S$ be a set of places consisting of primes above $p$, places where $\rhobar$ is ramified, and the archimedean places.  When using the local deformation conditions as above at $v \in S$, the inequality \eqref{eq:tspaceinequality} says exactly that
\begin{equation} \label{eq:finalequality}
 [K:\QQ] (\dim G_k - \dim B_k)  = \sum_{v | p} [K_v:\QQ_p] (\dim G_k - \dim B_k) \geq \sum_{v | \infty} h^0(\Gamma_v,\adzerorho) = \sum_{v | \infty} \adzerorho^{\Gamma_v}
\end{equation}
This is very strong: it is always true that $\dim \adzerorho^{\Gamma_v} \geq [K_v:\RR] (\dim G_k - \dim B_k)$, so \eqref{eq:finalequality} holds if and only if $K$ is totally real and $\rhobar$ is odd at all real places of $K$.

Assuming $K$ is totally real and $\rhobar$ is odd at all real places, we use Ramakrishna's deformation condition $\Dram_v$ at a collection of new places as in Proposition~\ref{prop:killingdualselmer} (again possibly extending $k$).  This gives a new deformation condition $\D_T$ for which $H^1_{\D_T^\perp}(\Gamma_T,\adzerorho^*)=0$.  Using Proposition~\ref{prop:localglobal}, we obtain the desired geometric lift.

Let us collect together all of our assumptions and record the result.  

\begin{thm} \label{thm:liftingapplication}
 Let $G = \GSp_{m}$ with even $m$ or $G = \GO_m$.   For a big representation $\rhobar : \Gamma_K \to G(k)$ with $p > m$, fix a lift $\nu : \Gamma_K \to (G/G')(k)$ to $\O$ of $\mu \circ \rhobar$ such that $\nu \otimes \O[\frac{1}{p}]$ is Fontaine-Laffaille.  We furthermore assume that $K$ is totally real and that $\rhobar$ is odd at all real places (which requires $m \not \equiv 2 \pmod{4}$ when $G = \GO_m$).  Assume that $p$ is unramified in $K$ and that $\rhobar$ is Fontaine-Laffaille at all places above $p$ with Fontaine-Laffaille weights in an interval of length $\frac{p-2}{2}$, pairwise distinct for each $\QQ_p$ embedding of $K$ into $\O[\frac{1}{p}]$.  Extend $\O$ (and $k$) so that all of the required local deformation conditions may be defined.   Then there is a finite set $T$ of places containing the archimedean places, the places above $p$, and the places where $\rhobar$ is ramified such that there exists a lift $\rho : \Gamma_K \to G(\O)$ such that
\begin{itemize}
 \item $\mu \circ \rho = \nu$;
 \item  $\rho$ is ramified only at places in $T$;
 \item  $\rho$ is Fontaine-Laffaille at all places above $p$, and hence crystalline.  
\end{itemize}
\end{thm}

In particular, $\rho$ is geometric.  If we combine this with Proposition~\ref{prop:bigimage}, we obtain Theorem~\ref{thm:maintheorem}.

\begin{remark}
The same argument works for $G= \GL_n$ using local deformation conditions like those of \cite[\S2.4.1]{cht08} and \cite[\S2.4.4]{cht08}.  The argument for $\GL_2$ is a variant of the proof \cite[Theorem 1b]{ram02}.    But for $n>2$ it is impossible to satisfy the oddness hypothesis.  To obtain representations that are odd one would need to work with $\GL_n \rtimes \on{Out}(\GL_n)$ or related groups, as is done for $\mathcal{G}_n$ in \cite[Theorem 2.6.3]{cht08}.
\end{remark}

\begin{remark}
 For other groups, the method will produce lifts provided appropriate local conditions exist.  The deformation conditions we used are only available in full strength for symplectic and orthogonal groups.  An alternative deformation condition above $p$ is the ordinary deformation condition \cite[\S4.1]{patrikis15}, available for any $G$.  For ramified primes not above $p$, \cite[\S5]{boohermr} provides a deformation condition assuming a certain nilpotent centralizer is smooth and $\rhobar|_{\Gamma_v}$ is tamely ramified.
\end{remark}

\section{Fontaine-Laffaille Theory with Pairings} \label{sec:fontainelaffaille}

We begin by establishing some notation and reviewing the key results of Fontaine-Laffaille theory.  It was first studied by Fontaine and Laffaille \cite{fl82}, who introduced a contravariant functor relating torsion-crystalline representations and Fontaine-Laffaille modules.  For deformation theory, in particular compatibility with tensor products, it is necessary to use a covariant version, introduced in \cite{bk90}.  
We then study Fontaine-Laffaille modules with the extra data of a pairing by analyzing tensor products and duals, in preparation for studying the Fontaine-Laffaille deformation condition in \S\ref{sec:fldeformation}.  This analysis generalizes unpublished results in \cite{patrikis06}.

\subsection{Covariant Fontaine-Laffaille Theory} \label{sec:covariantfl}

Let $K = W(k')[\frac{1}{p}]$ for a perfect field $k'$ of characteristic $p$.   Let $W = W(k')$ and $\sigma : W \to W$ denote the Frobenius map.  Recall that a \emph{torsion-crystalline} representation with Hodge-Tate weights in $[a,b]$ is a $\ZZ_p[\Gamma_K]$-module $T$ for which there exists a crystalline representation $V$ with Hodge-Tate weights in $[a,b]$ and  $\Gamma_K$-stable lattices $\Lambda \subset \Lambda'$ in $V$ such that $\Lambda'/ \Lambda$ is isomorphic to $T$.  Our convention will be that the Hodge-Tate weight of the cyclotomic character is $-1$, which will work well with covariant functors.
The analogue of torsion-crystalline representations on the semilinear algebra side are certain classes of Fontaine-Laffaille modules:

\begin{defn}
A \emph{Fontaine-Laffaille} module is a $W$-module $M$ together with a decreasing filtration $\{M^i\}_{i \in \ZZ}$ of $M$ by $W$-submodules and a family of $\sigma$-semilinear maps $\{ \varphi^i_M : M^i \to M\}$ such that:
\begin{itemize}
 \item The filtration is separated and exhaustive: $M = \cup_{i \in \ZZ} M^i$ and $\cap_{i \in \ZZ} M^i= 0$.
 \item  For $m \in M^{i+1}$, $p \cdot \varphi^{i+1}_M(m) = \varphi^i_M(m)$.
\end{itemize}
Morphisms of Fontaine-Laffaille modules $f : M \to N$ are $W$-linear maps such that $f(M^i) \subset N^i$ and $f \circ \varphi^i_M = \varphi^i_N \circ f$ for all $i$.  The category of Fontaine-Laffaille modules is denoted $\MF_W$.

Let $\MF_{W,\tor}^{f}$ denote the full subcategory consisting of $M$ for which $M$ is of finite length (as a $W$-module) and for which $\sum_{i \in \ZZ} \varphi^i(M^i) = M$, and $\MF_{W,\tor}^{f,[a,b]}$ to be the full subcategory with the additional condition that $M^a=M$ and $M^{b+1}=0$.
\end{defn}

Maps in $\MF_{W,\tor}^{f}$ are strict for the filtration, and $\MF_{W,\tor}^{f}$ is an abelian category.

\begin{remark}
Jumps in the filtration will turn out to correspond Hodge-Tate weights, so the condition $M^a=M$ and $M^{b+1} = 0$ with $a \leq b$ corresponds to Hodge-Tate weights lying in $[a,b]$.  We call the set of jumps in the filtration the \emph{Fontaine-Laffaille weights}.
\end{remark}

We are also interested in a variant that allows non-torsion modules.  

\begin{defn} \label{defn:dieudonne}
 A \emph{filtered Dieudonn\'{e}} module $M$ is a Fontaine-Laffaille module (that is finite over $W$) for which the $M^i$ are direct summands of $M$ as $W$-modules and for which 
 \[
   \sum_{i \in \ZZ} \varphi^i(M^i) = M.
 \]
 Let $\DK$ denote the full subcategory of $\MF_W$ consisting of filtered Dieudonn\'{e} modules $M$ for which there exist integers $a$ and $b$ for which $M^a = M$, $M^{b+1}=0$, and $0 \leq b-a \leq p-2$.
\end{defn}

Note that $\D_K$ is also an abelian category.  For $M \in \MF_{W,\tor}^{f,[a,b]}$, it is automatic that $M^i$ is a direct summand of $M$.  There are natural notions of tensor products and duality.

\begin{defn} \label{defn:fltensor}
For Fontaine-Laffaille modules $M_1$ and $M_2$, define $M_1 \tensor{W} M_2$ to have underlying $W$-module $M_1 \tensor{W} M_2$, filtration given by $(M_1 \tensor{W} M_2)^n = \sum_{i+j =n} M_1^i \tensor{W} M_2^j$, and maps $\varphi_{M_1 \tensor{W} M_2}^n $ induced by the $\varphi_{M_1}^i$ and $\varphi_{M_2}^j$.
\end{defn}

\begin{defn}\label{defn:mfdual}
For $M \in \MF_{W,\tor}^f$, define $M^*$ to be $\Hom_W(M,K/W)$ with the dual filtration $$(M^*)^i := \Hom_W(M/M^{1-i}, K/W)$$ and with $\varphi^i_{M^*}$ characterized for $f \in (M^*)^i$ and $m \in M^j$ by $\varphi^i_{M^*}(f) (\varphi^j(m))=0$ when $j \geq 1-i$ and by $\varphi^i_{M^*}(f) (\varphi^j(m)) = f(p^{-i-j}m)$ when $j < 1-i$ (in which case $-i-j \geq 0$).
\end{defn}

\begin{lem} \label{lem:mfdual}
There is a unique $(\varphi^i_{M^*})$ satisfying these constraints.  Using it, $M^*$ is an object of $\MF_{W,\tor}^f$. Then $M \mapsto M^*$ is a contravariant functor from $\MF^f_{W,\tor}$ to itself, and $M \simeq M^{* *}$ naturally in $M$.
\end{lem}

\begin{proof}
Uniqueness is immediate, while existence is checked in \cite[\S7.5]{conrad94}.  We will use a similar argument in Lemma~\ref{lem:flduality1} and Lemma~\ref{lem:flduality}.
\end{proof}

To connect Fontaine-Laffaille modules and torsion-crystalline representations, we use the period ring $\Acris$.  A convenient reference is \cite[\S2.2,2.3]{hattori}, which collects together previous work and reviews $\Acris$ for the purposes of constructing the contravariant and covariant Fontaine-Laffaille functors.  For our purposes, what is important is that $\Acris$ is a $W$-algebra that has an action of $\Gamma_K$, a $\sigma$-semilinear endomorphism $\varphi$ and a filtration $\{\Fil^i \Acris\}$.  In particular, it carries both an action of $\Gamma_K$ and the structure of a Fontaine-Laffaille module.   We use $\Acris$ to define an analogue of $V_\cris$:

\begin{defn} \label{defn:tcristor}
For $M \in \MF_{W,\tor}^{f,[2-p,1]}$, define 
\[
 T_\cris(M) := \ker \left(1 - \varphi^0_{\Acris \otimes M} : \Fil^0(\Acris \otimes M) \to \Acris \otimes M \right).
\]
\end{defn}

A small argument (see \cite[\S2.2]{hattori}) also shows that 
\[
 A_{\cris,\infty} := \Acris \tensor{W} K/W = \directlimit_n \Acris/p^n \Acris \in \MF_{W,\tor}^{f, [0,p-1]}.
\]
This allows us to define a contravariant functor from $\MF_{W,\tor}^{f, [0,p-1]}$ to $\Rep_{\ZZ_p}(\Gamma_K)$ by
\[
 T^*_{\cris}(M) := \Hom_{\MF_W}(M,A_{\cris,\infty}).
\]
This functor agrees with the functor $U_S$ considered by Fontaine and Laffaille \cite[Remark 2.7]{hattori}.

\begin{remark} \label{remark:tcrisstar}
If $M \in \MF_{W,\tor}^{f, [2-p,1]}$ is $p^n$-torsion, then
\begin{align*}
 T^*_{\cris}(M^*) &= \Hom_{\MF_W}(M^*,A_{\cris}/p^n \Acris) \\
 & \simeq \ker \left(1 - \varphi^0_{\Acris \otimes M} : \Fil^0(\Acris \otimes M) \to \Acris \otimes M \right) \\
 & =     T_\cris(M)
\end{align*}
which is how Fontaine and Laffaille's results about $T_\cris^*$ imply results about $T_\cris$. 
\end{remark}

We can extend $T_\cris$ to $\DK$ by defining an analogue of Tate-twisting.  

\begin{defn} 
 For $M \in \DK$ and an integer $s$, define $M(s)$ to have the same underlying $W$-module with filtration $M(s)^i = M^{i-s}$ and maps $\varphi_{M(s)}^i = \varphi_M^{i-s}$.  We then define $T_\cris(M)$ for a $M \in \DK$ that satisfies $M^a=M$ and $M^{b+1}=0$ for some $a, b \in \ZZ$ with $b-a \leq p-2$ by
 \[
 T_\cris(M) := T_\cris(M(-(b-1)))(b-1)
\]
\end{defn}

This agrees with Definition~\ref{defn:tcristor} on $\MF_{W,\tor}^{f,[2-p,1]}$ and also extends the definition to $\MF_{W,\tor}^{f,[a,b]}$ when $b-a \leq p-2$.

\begin{fact}\label{fact:tcris}
We have:
\begin{enumerate}
 \item The covariant functor $T_\cris : \DK \to \Rep_{\ZZ_p}[\Gamma_K]$ is well-defined, and is exact and fully faithful.  \label{tcris:1}
 
 \item  For $M \in \DK$, $T_\cris(M) = \inverselimit T_\cris(M/p^n M)$.\label{tcris:2}
 
 \item The essential image of $T_\cris : \MF_{W,\tor}^{f,[a,b]} \to \Rep_{\ZZ_p}[\Gamma_K]$ is stable under formation of sub-objects and quotients.\label{tcris:3}
 
 \item  For $M \in \MF_{W,\tor}^{f,[a,b]}$, the length of $M$ as a $W$-module is equal to the length of $T_\cris(M)$ as a $\ZZ_p$-module.  \label{tcris:4}
 
 \item  For $M \in \DK$, the $\Gamma_K$-representation $T_\cris(M) [\frac{1}{p}]$ is crystalline.\label{tcris:5}
 
 \item  Any torsion-crystalline $\FF_p[\Gamma_K]$-module $\Vbar$ whose Hodge-Tate weights lie in an interval of length $p-2$ is in the essential image of $T_\cris$.\label{tcris:6}
\end{enumerate}
\end{fact}

This is a modified version of \cite[Theorem 4.3]{bk90}.  
The first, fourth, and fifth statements are explicitly included in \cite[Theorem 4.3]{bk90}.  The second is checked in unpublished notes by Conrad \cite[\S7.2]{conrad94}; as we do not know a reference in the literature we include the argument in an appendix to the arXiv version of this article.  The claim about the essential image can be deduced from the analogous statement for $T_{\cris}^*$ using Remark~\ref{remark:tcrisstar}.  That statement is used in earlier work and checked in \cite[\S9.3]{conrad94}, and is explicitly stated and proven in \cite[Theorem 2.9(iv), \S4]{hattori}.
The last statement follows from relating $T_\cris$ to $T_\cris^*$ on $p$-torsion objects and the fact that for $r \in \{0,1,\ldots, p-2\}$, the functor $T_{\cris}^*$ induces an anti-equivalence between $\MF_{W,\tor}^{f, [0,r]}$ and the full subcategory of $\Rep_{\ZZ_p}(\Gamma_K)$ consisting of torsion-crystalline $\Gamma_K$ representations with Hodge-Tate weights in $[-r,0]$ (see for example \cite[Corollary 2.13]{hattori}).

\begin{remark} \label{rmk:weights}
 Our convention that the Hodge-Tate weight of the cyclotomic character is $-1$ makes the Fontaine-Laffaille weights and Hodge-Tate weights match under $T_\cris$. 
\end{remark}

\subsection{Tensor Products and Freeness} \label{sec:tensorfree}
Definition~\ref{defn:fltensor} defined a tensor product for Fontaine-Laffaille modules.  If $M_1 \in \MF_{W,\tor}^{f,[a_1,b_1]}$ and $M_2 \in \MF_{W,\tor}^{f,[a_2,b_2]}$, it is straightforward to verify that $M_1 \otimes M_2$ is an object of $\MF_{W,\tor}^{f,[a_1+a_2, b_1 + b_2]}$.  
The functor $T_\cris$ is compatible with tensor products in the following sense:

\begin{fact} \label{fact:tensor}
 Suppose that $M_1$, $M_2$, and $M_1 \otimes M_2$ each has Fontaine-Laffaille weights in an interval of length at most $p-2$.  Then the natural map $T_\cris(M_1) \tensor{\ZZ_p} T_\cris(M_2) \to T_\cris(M_1 \otimes M_2)$ is an isomorphism.
\end{fact}

The natural map comes from the multiplication of $A_{\cris}$.  To check this map is an isomorphism, one first checks it on simple $M_1$ and $M_2$ using Fontaine and Laffaille's classification of simple Fontaine-Laffaille modules when the residue field is algebraically closed.  Then one uses a d\'{e}vissage argument to reduce to the general case.  This argument comes from \cite{conrad94}, but as that reference is not publicly available, we sketch the argument in an appendix to the arXiv version of this article.

\begin{remark}
An analogue of this compatibility is stated in \cite[Remarques 6.13(b)]{fl82} for the contravariant functor $T^*_{\cris}$: the natural map
\[
 T_\cris^*(M_1) \tensor{\ZZ_p} T_\cris^*(M_2) \to T_\cris(M_1 \otimes M_2)
\]
is an isomorphism.  This statement is missing a $p$-torsion hypothesis, since there is no natural map in general.  When $M_1$ and $M_2$ are $p$-torsion, we have
\[
 T_\cris^*(M_1) = \Hom_{\MF_W}(M_1,A_{\cris,\infty}) = \Hom_{\MF_W}(M_1,A_{\cris}/p A_{\cris})
\]
and likewise for $M_2$.  Then multiplication on $A_{\cris}/p A_{\cris}$ gives a natural map $$T_\cris^*(M_1) \otimes T_\cris^*(M_2) \to T_\cris^*(M_1 \otimes M_2)$$ which can be checked to be an isomorphism by d\'{e}vissage.  But $A_{\cris,\infty}$ is not a ring, so there is no natural map without a $p$-torsion hypothesis on $M_1$ and $M_2$.  This explains why it is crucial to work with the covariant functor $T_\cris$.  
\end{remark}

For $M \in \MF_{W,\tor}^{f,[a,b]}$, if $V=T_\cris(M)$ has ``extra structure'' then so does $M$.  For example, if $V$ were a deformation of a residual representation over a finite field $k$, $V$ would be an $\O = W(k)$-module.  As $T_\cris$ is covariant and fully faithful, it is immediate that $M$ is naturally an $\O$-module.  The actions of $\ZZ_p$ on $M$ via the embeddings into $\O$ and $W=W(k')$ are obviously compatible.  We denote the Frobenius on $\O$ by $\sigma$.

Recall that Galois representations of $\Gamma_K$ defined over a finite extension $L$ of $\QQ_p$ can be viewed as $\QQ_p$-vector spaces with the additional action of $L$.  
Assume there exists an embedding of $K$ into $L$ over $\QQ_p$, so $L$ splits $K$ over $\QQ_p$.  Such Galois representations are modules over $L \tensor{\QQ_p} K \simeq \prod_{\tau : K \into L} L_\tau$ via $a \otimes b \mapsto (a \tau(b))$.  For each $\QQ_p$-embedding $\tau$, there is a collection of Hodge-Tate weights.  We will generalize this structure to the setting of Fontaine-Laffaille modules.  

Assume $k'$ is finite, and more specifically that $k'$ embeds in $k$, so $\O[ \frac{1}{p}]$ splits the finite unramified $K$ over $\QQ_p$.  Hence
\[
 \O \tensor{\ZZ_p} W \simeq \prod_{\tau : W \into \O} \O_\tau
\]
as $\O$-algebras, where $\tau$ varies over $\ZZ_p$-embeddings of $W$ into $\O$ and $W$ acts on $\O_\tau := \O$ via $\tau$.  We likewise obtain a decomposition of the $\O \tensor{\ZZ_p} W$-module $M$ as
\[
 M = \bigoplus_{\tau : W \into \O} M_\tau.
\]
Note that 
\[
 \Hom_{\O \tensor{\ZZ_p} W} (M,M') = \bigoplus_{\tau : W \into \O} \Hom_\O(M_\tau,M'_\tau).
\]

\begin{lem} \label{lem:ostructure}
If $V = T_\cris(M)$ is equipped with a $\Gamma_K$-equivariant $\O$-module structure then for $M_\tau^i := M_\tau \cap M^i$ we have
\[
 M^i = \bigoplus_{\tau : W \into \O} M^i_\tau
\]
and furthermore the $\sigma$-semilinear map $\varphi^i_M |_{M^i_\tau} : M_\tau^i \to M$ factors through $M_{\sigma \tau}$.  The length of $M$ as an $\O$-module equals the length of $V$ as an $\O$-module multiplied by $[K:\QQ_p]$.
\end{lem}

\begin{proof}
 The first statement is straightforward, and the second is bookkeeping using Fact~\ref{fact:tcris}\eqref{tcris:4}.
\end{proof}
%
%
%
%

We also have a result about freeness.

\begin{lem} \label{lem:rfree}
Let $V = T_\cris(M)$ and $R$ be an artinian coefficient $\O$-algebra $R$ with residue field $k$.  Then $M$ is a $R$-module object in $\MF_{W,\tor}^{f,[a,b]}$ if and only if $V$ is an $R$-module object in $\Rep_{\ZZ_p}[\Gamma_K]$.  In that case, $V$ is a free as an $R$-module if and only if $M$ is free as an $R$-module.  When $M$ is free, all of the $M_\tau^i$ are free $R$-direct summands.  All of the $M_\tau$ have the same rank.
\end{lem}

\begin{proof}
The full faithfulness of $T_\cris$ allows the transport of $R$-module structure.
Let $N$ be a finitely generated $R$-module with $n = \dim_k N / \m_R N$.  Then $N$ is free if and only if $\lg_\O(N) = n \lg_\O (R)$ , as we see via Nakayama's lemma applied to a map $R^n \to N$ inducing an isomorphism modulo $\m_R$.  From the exact sequence of Fontaine-Laffaille modules
\[
 0 \to \m_R M \to M \to M/ \m_R M \to 0
\]
and the fact that $T_\cris$ is covariant and exact, we see that $T_\cris(M/\m_R M) = V /\m_R V$.  Using Lemma~\ref{lem:ostructure}, if $\dim_k V /\m_R V =n$ then $M / \m _R M$ is a $k$-vector space of dimension $[K:\QQ_p] n$.  Thus to relate $R$-freeness of $M$ and $V$ we just need to show that $\lg_\O(M) = [K:\QQ_p] \lg_\O(V)$, which again follows from Lemma~\ref{lem:ostructure}.

Now suppose $M$ is a free $R$-module.  By functoriality, the $\ZZ_p$-module direct summands $M_\tau$ of $M$ are each $R$-submodules, so each $M_\tau$ is an $R$-module direct summand of $M$.  Hence each $M_\tau$ is $R$-free when $M$ is free.  To deduce the same for each $M_\tau^i$, we just need that each $M_\tau^i$ is an $R$-module summand.  By $R$-freeness of $M$, it suffices to show that each $M_\tau^i  / \m_R M_\tau^i \to M / \m_R M$ is injective.  Since $M_\tau^i$ is the ``$\tau$-component'' of $M^i$ by Lemma~\ref{lem:ostructure} it is an $R$-module summand of $M^i$.  Thus it suffices to show that 
\[
 M^i / \m_R M^i \to M/ \m_R M
\]
is injective for all $i$.  But this follows from the fact that morphisms in $\MF_{W,\tor}^{f,[a,b]}$ are strict.

To check that all of the $M_\tau$ have the same rank, by freeness it suffices to check that $\dim_k \overline{M}_\tau$ is independent of $\tau$.  As all $\ZZ_p$-embeddings of the \emph{unramified} $W$ into $\O$ are of the form $\sigma^i \tau$ for some fixed $\ZZ_p$-embedding $\tau$ and $\sigma$ has finite order, it suffices to show that
\[
 \dim_k \overline{M}_\tau \geq \dim_k \overline{M}_{\sigma \tau}.
\]
As each $\overline{M}^i_\tau$ is a $k$-module direct summand of $\overline{M}_\tau$, $\overline{M}_\tau$ is isomorphic to $\gr^\bullet \overline{M}_\tau$.  But $\varphi^i_{\Mbar}(\overline{M}^{i+1}) = 0$, so we obtain a map
\[
 \sum_i \varphi^i_{M_\tau} : \gr^\bullet \overline{M}_\tau \to \overline{M}_{\sigma \tau}.
\]
 As Fontaine-Laffaille modules satisfy                                                                                                                                                                                                    
\[
 \overline{M} = \sum_{i} \varphi^i(\overline{M}^i)
\]
the map $\sum_i \varphi^i_{M_\tau}$ is surjective.  This completes the proof.
\end{proof}

\begin{remark} \label{remark:embeddingweights}
We get a set of Fontaine-Laffaille weights for each $\ZZ_p$-embedding $\tau : W \into \O$.  We can also define the multiplicity of a weight $w_\tau$ to be the rank of the $R$-module $M^{w_\tau}_\tau / M^{w_\tau+1}_\tau$.  The number of Fontaine-Laffaille weights (counted with multiplicity) is the same for each embedding.  We say the Fontaine-Laffaille weights with respect to an embedding are distinct if each has multiplicity $1$.  This is analogous to the way a Hodge-Tate representation of $\Gamma_K$ over a $p$-adic field splitting $K$ over $\QQ_p$  has a set of Hodge-Tate weights for each $\QQ_p$-embedding of $K$ into that field.  
\end{remark}

We can now define a notion of a tensor product for Fontaine-Laffaille modules that are also $R$-modules objects for a coefficient ring $R$ over $\O$.

\begin{defn}
 Define $M_1 \tensor{W \tensor{\ZZ_p} R} M_2$ to be the module $M_1 \tensor{W \tensor{\ZZ_p} R} M_2$ together with filtration defined by $(M_1 \tensor{W \tensor{\ZZ_p} R} M_2)^n = \sum_{i+j = n} M_1^i \tensor{W \tensor{\ZZ_p} R} M_2^j$ and with $\varphi^n_{M_1 \tensor{W \tensor{\ZZ_p} R} M_2}$ defined in the obvious way on the pieces.
\end{defn}

\begin{lem} \label{lem:fltensorr}
Suppose that $M_1$ and $M_2$ are $R$-module objects for a coefficient ring $R$ over $\O$ and that $M_1$, $M_2$, and $M_1 \tensor{W \tensor{\ZZ_p} R} M_2$ are all in $\MF_{W,\tor}^{f,[a,b]}$.  The natural map $T_\cris(M_1) \tensor{R} T_\cris(M_2) \to T_\cris(M_1 \tensor{W \tensor{\ZZ_p} R} M_2)$ is an isomorphism of $R[\Gamma_K]$-modules.
\end{lem}

\begin{proof}
We have an exact sequence
\[
 0 \to J \to M_1 \tensor{W} M_2 \to M_1 \tensor{W \tensor{\ZZ_p} R} M_2 \to 0
\]
where $J$ is generated by the extra relations imposed by $R$-bilinearity (beyond $W$-bilinearity).  For $r \in R$, define $\mu_r : M_1 \tensor{W} M_2 \to M_1 \tensor{W} M_2$ by
\[
 \mu_r( m_1 \otimes m_2) = (r m_1) \otimes m_2 - m_1 \otimes (r m_2).
\]
Then $J = \sum_{r \in R} \Im(\mu_r)$; this is an object in the abelian category $\MF_{W,\tor}^{f,[a,b]}$.  We will show that $T_\cris(J)$ is the kernel of $T_\cris(M_1 \tensor{W} M_2) \to T_\cris(M_1) \tensor{R} T_\cris(M_2)$.

It suffices to show that $T_\cris(N_1 + N_2) = T_\cris(N_1) + T_\cris(N_2)$ for subobjects $N_1$ and $N_2$ of $M_1 \tensor{W} M_2$.  Indeed, granting this we would know that
\[
 T_\cris(J) = \sum_{r \in R} T_\cris(\mu_r).
\]
But by functoriality $T_\cris(\mu_r)$ is the map $T_\cris(M_1) \tensor{W} T_\cris(M_2) \to T_\cris(M_1) \tensor{W} T_\cris(M_2)$ given by $v_1 \otimes v_2 \mapsto r v_1 \otimes v_2 - v_1 \otimes r v_2$, so $T_\cris(J)$ is the kernel of $T_\cris(M_1 \tensor{W} M_2) \to T_\cris(M_1) \tensor{R} T_\cris(M_2)$ as desired.

To prove that $T_\cris(N_1 + N_2) = T_\cris(N_1) + T_\cris(N_2)$, consider the exact sequence
\[
  0 \to N_1 \cap N_2 \to N_1 \oplus N_2 \to N_1 + N_2 \to 0.
\]
As $T_\cris$ preserves direct sums, it suffices to show that
\[
 T_\cris(N_1) \cap T_\cris(N_2) = T_\cris(N_1 \cap N_2).
\]
But this follows from the exactness of $T_\cris$ and the left exact sequence
\[
 0 \to N_1 \cap N_2 \to N_1 \oplus N_2 \to M_1 \tensor{W} M_2
\]
where the second map is $(n_1, n_2) \mapsto n_1 - n_2$.  
\end{proof}

\subsection{Duality}
Let $R$ be a coefficient ring over $\O$ and $M \in \MF_{W,\tor}^f$ be a free $R$-module compatible with the Fontaine-Laffaille structure in the sense that the action of $R$ is given by morphisms of Fontaine-Laffaille modules.  Fix $L \in \MF_{W,\tor}^f$ with an $R$-structure compatible with the Fontaine-Laffaille structure so that for each $\tau$, $L_\tau$ is a free $R$-module of rank $1$ with $L^{s_\tau}_\tau = L_\tau$ and $L^{s_\tau + 1}_\tau =0$ for some $s_\tau$ (the analogue of a character taking values in $R^\times$).  We will define a dual relative to $L$ akin to Cartier duality.  This will be useful for studying pairings.

\begin{defn} \label{def:mfrdual}
For an $M$ as above, define $M^\vee = \Hom_{R \tensor{\ZZ_p} W} (M,L)$ with a filtration given by 
\[
 \Fil^i M^\vee = \{ \psi \in \Hom_{R \tensor{\ZZ_p} W}(M,L) : \psi(M^j) \subset L^{i+j} \, \mbox{ for all } j \in \ZZ\}.
\]
For $\psi \in \Fil^i M^\vee$, define $\varphi^i_{M^\vee}(\psi)$ to be the unique function in $\Hom_{R \tensor{\ZZ_p} W}(M,L)$ such that 
\[
 \varphi^i_{M^\vee}(\psi) (\varphi^j_M(m)) = \varphi_L^{i+j} (\psi( m))
\]
for all $m \in M^j$ and $j$.
\end{defn}

If $\varphi^i_{M^\vee}$ exists, it is unique since the images of the $\varphi^j_M$'s span $M$ additively.  Likewise, if $\varphi^i_{M^\vee}$ exists for all $i$ they are automatically $\sigma$-semilinear and satisfy $p \varphi^{i+1}_{M^\vee} = \varphi^i_{M^\vee} |_{\Fil^{i+1} M^\vee}$.  We check $\varphi^i_{M^\vee}(\psi)$ is well-defined in the following lemma.  The key fact is that all of the $M_\tau^i$ are free $R$-module direct summands of $M_\tau$ (by Lemma~\ref{lem:rfree}).

\begin{lem} \label{lem:flduality1}
 The function $\varphi^i_{M^\vee}(\psi)$ is well-defined, and the filtration can equivalently be described as
\[
 \Fil^i M^\vee = \bigoplus_{\tau : W \into \O} \Hom_{R} (M_\tau/M_\tau^{1+s_\tau-i} , L_\tau).
\]
\end{lem}

\begin{proof}
We first establish the alternate description of $\Fil^i M^\vee$.   Because
\[
 \Hom_{R \tensor{\ZZ_p} W}(M,L) = \bigoplus_{\tau : W \into \O} \Hom_R (M_\tau, L_\tau),
\]
and $L_\tau^{s_\tau} = L_\tau$ while $L_\tau^{s_\tau+1} = 0$, an element $\psi_\tau \in \Hom_R(M_\tau,L_\tau)$ satisfies $\psi_\tau(M_\tau^j) \subset L_\tau^{i+j}$ if and only if $\psi_\tau(M^j_\tau) =0$ whenever $i+j > s_\tau$.  This says exactly that $\psi_\tau$ factors through $M_\tau/M_\tau^{1+s_\tau-i}$.  Because $M^{1+s_\tau-i}_\tau$ is an $R$-module direct summand, hence free with free complement, a morphism $M_\tau/M_\tau^{1+s_\tau-i} \to L_\tau$ is equivalent to a morphism $\psi_\tau: M_\tau \to L_\tau$ such that $\psi_\tau(M^{1+s_\tau -i}_\tau) =0$.  Thus $\Fil^i M^\vee_\tau = \Hom_R(M_\tau/ M_\tau^{1+s_\tau -i},L_\tau)$ as desired.

We will construct $\varphi^i_{M^\vee} : \Fil^i M^\vee \to M^\vee$ using the exact sequence
\begin{align} \label{eq:flexact}
 0 \to \bigoplus_{r = a+1}^b M^r \to \bigoplus_{r=a}^b M^r \to M \to 0
\end{align}
of \cite[Lemme 1.7]{fl82}.  The first map sends $(m_r)_{r=a+1}^{r=b}$ to $(p m_r - m_{r+1})_{r=a}^{r=b}$ (with the convention that $m_a =0$ and $m_{b+1}=0$), and the second map is $\sum_{r=a}^b \varphi^r_M$.  For $\psi \in \Fil^i M^\vee$, consider the map 
\[
 \phi : \bigoplus_{r=a}^b M^r \to L
\]
induced by the $\varphi^{i+r}_{L} \circ \psi : M^r \to L$.  For $(m_r)_{r=a+1}^{r=b}$ in $\displaystyle \bigoplus_{r = a+1}^b M^r$, we compute that
\begin{align*}
\phi( (m_r)_{r=a+1}^{r=b}) & = \sum_{j=a}^b \varphi_{L}^{i+j} (\psi( (p m_j - m_{j+1}) )) \\
 &= \sum_{j=a}^b p \varphi_{L}^{i+j} (\psi( m_j)) - \sum_{j=a}^b  \varphi_{L}^{i+j} (\psi( m_{j+1})).
\end{align*}
But $\varphi_{L}^{i+j} |_{L^{i+j+1}} = p \varphi_{L}^{i+j+1}$, so this difference is
\[
 \sum_{j=a}^b p \varphi_{L}^{i+j} (\psi( m_j)) - \sum_{j=a+1}^{b+1}  p\varphi_{L}^{i+j} (\psi( m_{j}))
\]
which vanishes as $m_{b+1}=0$ and $m_a=0$.
Hence $\phi$ factors through the quotient $M$ of \eqref{eq:flexact}, giving the desired well-defined map $\varphi^i_{M^\vee}$.  
\end{proof}

\begin{lem} \label{lem:flduality}
The Fontaine-Laffaille module $M^\vee$ is an object of $\MF_{W,\tor}^{f}$.  
\end{lem}

\begin{proof}
It suffices to show that the inclusion
\[
 \sum_i \varphi^i_{M^\vee} (\Fil^i M^\vee) \into M^\vee
\]
is an equality.  By Nakayama's lemma, it suffices to show that the reduction modulo $\m_R$ is surjective.  For an $R$-module $N$, let $\overline{N}$ denote the reduction modulo $\m_R$.  We may pick free $R$-modules $N_\tau^i$ such that $M^i_\tau = N_\tau^i \oplus M^{i+1}_\tau$ as each $M^i_\tau$ is a (free) direct summand of the $R$-free $M_\tau$ that is an $R$-free direct summand of $M$.  Because $p \cdot \varphi^{i+1}_M = \varphi^i_M |_{M^{i+1}}$, we see $\varphi^i_{M} ( \overline{M}^i_\tau) = \varphi^i_M(\overline{N}^i_\tau)$, so 
\[
 \overline{M}_{\sigma \tau} = \sum_i \varphi^i_M(\overline{N}^i_\tau).
\]
By Lemma~\ref{lem:rfree}, $\overline{M}_\tau$ and $\overline{M}_{\sigma \tau}$ have the same dimension so $\varphi^i_M |_{N_\tau^i}$ is injective and the sum is direct.  We also know that $\varphi_L^{i} |_{\overline{L_\tau}} =0$ for $i < s_\tau$ because $p \cdot \varphi^{j+1}_L = \varphi^j_L |_{L^{j+1}}$.

As $M_\tau$ and $L_\tau$ are free $R$-module summands of $M$ and $L$ for all $\tau$, $\overline{M^\vee} = \overline{M}^\vee$ by Lemma~\ref{lem:flduality1}.
We can describe an element $\psi \in \Fil^i \overline{M^\vee}$ as a collection of $ \psi_{\tau,j} \in \bigoplus_{\tau,j} \Hom_{R}(\overline{N}_\tau^j , \overline{L}^{i+j}_\tau)$.  But $\overline{L}^{i+j}_\tau$ is one-dimensional over $k$ if $i+j \leq s_\tau$, and is zero otherwise.    
Then for $f = \varphi^i_{M^\vee}(\psi)$ and $m = \sum_{\tau,j} \varphi^j_M( n_{\tau,j})$ with $n_{\tau,j} \in \overline{N}_{\tau}^j$, by construction we have
\[
 f(m) = \sum_{\tau,j} \varphi_L^{i+j} ( \psi(n_{\tau,j})).
\]
But $\varphi_L^{i+j} ( \psi(n_{\tau,j}))$ is forced to be zero unless $i+j = s_\tau$, in which case it can take on any non-zero value in $\overline{L}_\tau$ (depending on the choice of $\psi$).  Thus
\[
 \varphi^i_{M^\vee} (\Fil^i \overline{M}^\vee) = \bigoplus_\tau \Hom \left(\varphi^{s_\tau -i}_M(\overline{N}_{\tau}^{s_\tau-i}),\overline{L}_{\sigma \tau} \right).
\]
Summing over $i$, and using the sum decomposition $\overline{M} = \sum_{\tau, i} \varphi^i_M(\overline{N}^i_\tau)$ gives that
\[
 \sum_i \varphi^i_{\overline{M}^\vee} (\Fil^i \overline{M}^\vee) = \Hom( \overline{M} , \overline{L}).
\]
This shows the desired surjectivity.
\end{proof}

\begin{remark}
 For fixed $\ZZ_p$-embedding $\tau : W \into \O$, if the Fontaine-Laffaille weights (Remark~\ref{remark:embeddingweights}) of $M$ with respect to $\tau$ are $\{w_{\tau,i}\}_i$ then the Fontaine-Laffaille weights of $M^\vee$ with respect to $\tau$ are $\{s_\tau - w_{\tau,i}\}_i$.
\end{remark}

Now assume we have a Galois representation $\nu$ on the free rank-$1$ $R$-module corresponding to $L$; we define the \emph{dual} $V^\vee = \Hom_{R[\Gamma_K]}(V,R(\nu))$ for a discrete $\Gamma_K$-representation on a finite free $R$-module $V$.

\begin{lem} \label{lem:compatibleduals}
For a morphism $f: M \to N$ in $\MF_{W,\tor}^{f,[a,b]}$ with $b-a \leq \frac{p-2}{2}$, there is a natural isomorphism $T_\cris(M^\vee) \simeq T_\cris(M)^\vee$ and $T_\cris(f^\vee) = T_\cris(f)^\vee$.
\end{lem}

\begin{proof}
 We prove this by studying the evaluation pairing $M \tensor{R} M^\vee \to L$.  It is straightforward to verify that this pairing is a morphism of Fontaine-Laffaille modules.  Because $b-a \leq \frac{p-2}{2}$, Lemma~\ref{lem:fltensorr} gives a pairing of Galois-modules
 \begin{equation} \label{eq:dualitypairing}
 T_\cris(M) \tensor{R} T_\cris(M^\vee) = T_\cris(M \tensor{W \tensor{\ZZ_p} R} M^\vee)  \to T_\cris(L) = \nu.
 \end{equation}
 
We will now prove that this pairing is perfect when $R = k$.  We will do so by inducting on the dimension of the $k$-vector space $M$.  The case of dimension $0$ is clear.  If $M \neq 0$ the pairing of Fontaine-Laffaille modules is non-zero (look at the pairing $M_\tau \times \Hom(M_\tau,L_\tau) \to L_\tau$ of vector spaces).  Thus the pairing of Galois-modules is non-zero if $M \neq 0$ as $T_\cris$ is faithful.  

Now we use induction, so we can assume $M \neq 0$.  The annihilator of $T_\cris(M^\vee)$ is $T_\cris(M_1)$ for some $f : M_1 \into M$ because the essential image of $T_\cris$ is closed under taking sub-objects.  We know $M_1$ is a proper sub-object as the pairing is non-zero.  Observe that we may define the dual $f^\vee : M^\vee \to M_1^\vee$ by precomposition: it is surjective as we are over a field.  For $v_1 \in T_\cris(M_1)$ and $v_2 \in T_\cris(M^\vee)$, we must have
\[
 0 =\langle v_1, f^\vee v_2 \rangle = \langle f(v_1), v_2 \rangle.
\]
But the pairing $T_\cris(M_1) \otimes T_\cris(M_1^\vee) \to T_\cris(L)$ is non-degenerate by induction, and $f^\vee$ is surjective, so this means that $v_1=0$.  Thus $T(M_1)$ and hence $M_1$ are trivial.  Over the field $k$, this ensures the pairing is perfect.

For the general case, we use the basic fact that for a coefficient ring $R$, if $N_1$ and $N_2$ are free $R$-modules of the same rank with an $R$-bilinear pairing $N_1 \times N_2 \to R$, the pairing is perfect if the reduction (modulo $\m_R$) $\overline{N_1} \times \overline{N_2} \to k$ is perfect.  
Apply this to $T_\cris(M) \times T_\cris(M^\vee) \to T_\cris(L)$.

The statement $T_\cris(f^\vee) = T_\cris(f)^\vee$ is just functoriality.
\end{proof}

\section{Fontaine-Laffaille Deformations} \label{sec:fldeformation}

Let $G = \GSp_{r}$ or $\GO_{r}$, and consider a representation $\rhobar : \Gamma_K \to G(k)$ with similitude character $\nubar$, where $K = W[\frac{1}{p}]$ for $W = W(k')$ with finite $k'$.  Let $\overline{V}$ be the underlying vector space for $\rhobar$ using the standard representation of $G$.  Take $\O$ to be the Witt vectors of $k$, and assume $\O[\frac{1}{p}]$ splits $K$ over $\QQ_p$.  Fix a lift $\nu : \Gamma_K \to \O^\times$ of $\nubar$ that is crystalline with Hodge-Tate weights $\{s_\tau\}_\tau$ in an interval of length $p-2$, and let $T_\cris(L) = \nu$.  

We suppose that $\rhobar$ is torsion-crystalline with Hodge-Tate weights in an interval $[a,b]$ where $0\leq b-a \leq \frac{p-2}{2}$ so we can use Fontaine-Laffaille theory.  Let $\Mbar$ be the corresponding Fontaine-Laffaille module (using Fact~\ref{fact:tcris}\eqref{tcris:6}), with Fontaine-Laffaille weights $\{w_{\tau,i} \}_{\tau,i}$.  In this section we define and study the Fontaine-Laffaille deformation condition assuming that for each $\ZZ_p$-embedding $\tau : W \into \O$ the Fontaine-Laffaille weights are multiplicity-free as in Remark~\ref{remark:embeddingweights} (the jumps in the filtration are of rank $1$).  This section is a generalization of unpublished results in \cite{patrikis06}, which treat the symplectic case when $K = \QQ_p$.

\subsection{Definitions and Basic Properties}
As $\Vbar$ is a $k$-linear representation of $\Gamma_K$, $\Mbar$ becomes a $k' \tensor{\FF_p} k$-module and in particular a $k$-vector space.  

\begin{defn}
For an Artinian coefficient ring $R$ over $\O = W(k)$, define $\DFL_\rhobar(R)$ to be the collection of deformations $\rho : \Gamma_K \to G(R)$ of $\rhobar$ with similitude character $\nu$ that lie in the essential image of $T_\cris$ (after composing with $G \to \GL_n$) restricted to the full subcategory $\MF_{W,\tor}^{f,[a,b]}$.
Such a deformation is called a \emph{Fontaine-Laffaille deformation}.
\end{defn}

This is a deformation condition: see Corollary~\ref{cor:isdefcond}.  We will analyze it when for each fixed embedding $\tau : W \into \O$ the Fontaine-Laffaille weights of $\rhobar$ are multiplicity-free  (when the jumps in the filtration of each $\Mbar_\tau$ are $1$-dimensional over $k$).  Note that the Fontaine-Laffaille weights of $M$ are the same as the Fontaine-Laffaille weights of $\Mbar$ as each $M^i_\tau$ is a direct summand.

\begin{thm} \label{thm:fl}
If the Fontaine-Laffaille weights are multiplicity-free, $\DFL_\rhobar$ is liftable.  If $B$ is a Borel subgroup of $G$, the dimension of the tangent space of $\DFL_\rhobar$ is
 \[
  [K:\QQ_p] \left( \dim G_k - \dim B_k \right) + H^0(\Gamma_K,\adzerorho).
 \]
If $\rho : \Gamma_K \to G(\O)$ is an inverse limit of Fontaine-Laffaille deformations of $\rhobar$ to $\O/p^n \O$ for all $n\geq 1$, it is a lattice in a crystalline representation with the same Fontaine-Laffaille weights as $\rhobar$.
\end{thm}

The proof of this theorem will occur over the remainder of this section.  The key pieces are Proposition~\ref{prop:charzero}, Proposition~\ref{prop:flliftable}, and Proposition~\ref{prop:fltangent}.

%
%

To understand $\DFL_{\rhobar}$, we must express the orthogonal or symplectic pairing in the language of Fontaine-Laffaille modules.  For a Galois module $V$ which is a free $R$-module, recall we defined $V^\vee = \Hom_{R[\Gamma_K]}(V,\nu)$.  For a deformation of $\rhobar$ to a coefficient ring $R$, we obtain an $R[\Gamma_K]$-module $V$ together with an isomorphism $\eta : V \simeq V^\vee$ coming from the pairing.   Let $\epsilon = 1$ for $G = \GO_r$ and $\epsilon = -1$ for $G = \GSp_{r}$.  The fact that $ \langle v,w \rangle = \epsilon \langle w,v\rangle$ is equivalent to $\eta^* = \epsilon \eta$, where $\eta^*$ is the map $V \simeq V^{\vee \vee}  \to V^\vee$ induced by double duality.

\begin{lem}
For a coefficient ring $R$, suppose $V$ is a lift of $\Vbar$ as an $R[\Gamma_K]$-module that is finite free over $R$, corresponding to a Fontaine-Laffaille module $M$ that is finite free over $R$.  An isomorphism of $R[\Gamma_K]$-modules \[\eta: V \simeq V^\vee\] such that $\eta(v)(w) = \epsilon \eta(w)(v)$ is equivalent to an $R$-linear isomorphism of Fontaine-Laffaille modules
\[
 \gamma : M \simeq M^\vee
\]
such that $\gamma(m)(n) = \epsilon \gamma(n)(m)$.
\end{lem}

\begin{proof}
As the Hodge-Tate weights of $\rhobar$ lie in an interval of length $\frac{p-2}{2}$, Lemma~\ref{lem:fltensorr} and Lemma~\ref{lem:compatibleduals} hold.  In particular, $T_\cris(M^\vee) = T_\cris(M)^\vee$.  As $T_\cris$ is fully faithful in this range, we see that a map $\eta$ is equivalent to a map $\gamma$, and one is an isomorphism if and only if the other one is.  It remains to check that $\gamma$ is symmetric or alternating if and only if $\eta$ is.  Let $\eta^*$ and $\gamma^*$ denote the isomorphisms respectively given by
\[
 V \simeq V^{\vee \vee} \overset{\eta^\vee} \to V^\vee \quad \text{and}  \quad M \simeq M^{\vee \vee} \overset{\gamma^\vee} \to M^\vee.
\]
A straightforward check shows that $T_\cris$ carries $\eta^*$ to $\gamma^*$, and hence $\eta = \epsilon \eta^*$ if and only if $\gamma = \epsilon \gamma^*$.
\end{proof}

\begin{lem}
An $R$-linear isomorphism of Fontaine-Laffaille modules $\gamma : M \simeq M^\vee$ for which $\gamma(m)(n)  = \epsilon \gamma(n)(m)$ is equivalent to a perfect $\epsilon$-symmetric $W\tensor{\ZZ_p} R$-bilinear pairing $\langle \cdot , \cdot \rangle : M \times M \to L_R$ satisfying
\begin{itemize}
 \item $\langle M^i, M^j \rangle \subset L^{i+j }$;
 \item $\langle \varphi^i_M(m), \varphi^j_M(n) \rangle = \varphi^{i+j}_L \langle m,n\rangle$.
\end{itemize}
\end{lem}

\begin{proof}
This is just writing out what $\gamma : M \to M^\vee$ being a morphism of Fontaine-Laffaille modules means for the pairing $\langle m,n \rangle = \gamma(m)(n)$.  

For $\gamma$ to preserve the filtration says exactly that 
\[
 \gamma(M^i ) \subset \Fil^i M^\vee = \{ \psi \in \Hom_{W \tensor{\ZZ_p} R}(M,L) : \psi(M^j) \subset L^{i+j} \}.
\]
This is equivalent to $\langle M^i, M^j \rangle \subset L^{i+j}$ for all $i,j$.  The compatibility of $\gamma$ with the $\varphi$'s says exactly that for $m \in M^i$
\[
 \varphi_{M^\vee}^i(\gamma(m)) = \gamma( \varphi^i_M(m)).
\]
Evaluating on any $\varphi^j_M(n) \in M$ and using the definition of $M^\vee$ we see 
\[
\varphi_{M^\vee}^i(\gamma(m))(\varphi^j_M(n))  = \varphi^{i+j}_L(\gamma(m)(n)) = \varphi^{i+j}_L(\langle m,n \rangle).
\]
Evaluating $\gamma( \varphi^i_M(m))$, we see that
\[
 \gamma( \varphi^i_M(m))(\varphi^j_M(n)) = \langle \varphi^i_M(m) , \varphi^j_M(n) \rangle.
\]
Thus, $\gamma$ being compatible with the $\varphi$'s is equivalent to $\langle \varphi^i_M(m) , \varphi^j_M(n) \rangle = \varphi^{i+j}_L(\langle m,n \rangle)$.
\end{proof}

In particular, the pairing $\Vbar \times \Vbar \to \nubar$ gives a perfect pairing $\langle \cdot , \cdot \rangle_{\Mbar} : \Mbar \times \Mbar \to \overline{L}$ .  

\begin{cor} \label{cor:pairingcompat}
For a coefficient ring $R$, $T_{\cris}$ gives a bijection between deformations $\rho \in \DFL_\rhobar(R)$ and isomorphism classes of Fontaine-Laffaille modules $ M \in \MF_{W,\tor}^{f,[a,b]}$ that are free as $R$-modules and for which there exists a perfect $\epsilon$-symmetric $W \tensor{\ZZ_p} R$-bilinear pairing $\langle \cdot , \cdot \rangle : M \times M \to L_R$ satisfying
\begin{itemize}
 \item $\langle M^i, M^j \rangle \subset L^{i+j }$;
 \item $\langle \varphi^i_M(m), \varphi^j_M(n) \rangle = \varphi^{i+j}_L \langle m,n\rangle$
\end{itemize}
together with an isomorphism of the reduction of $(M,\langle \cdot, \cdot \rangle)$ with $(\Mbar, \langle \cdot , \cdot \rangle_{\Mbar})$.
\end{cor}

\begin{proof}
This essentially follows by combining the two previous lemmas.  
Note that the pairing $\langle \cdot ,\cdot \rangle$ is automatically perfect as it lifts the perfect pairing $\langle \cdot , \cdot \rangle_{\Mbar}$.  One subtle point is that given such an $M$ with a pairing, we obtain an $\epsilon$-symmetric pairing on the corresponding $V$, but this pairing might not be the one used to define $G$ so the representation would not take values in $G(R)$.  However, after conjugation by an element of $\GL_r(R)$ that reduces to the identity modulo $\m_R$ the pairings will agree.

To show this, pick a basis and suppose that $J$ and $J'$ are matrices for $\epsilon$-symmetric pairings over $R$ that are equal modulo $R/I$, where $R \to R/I$ is a small extension and $I$ is dimension $1$ as a module over $R/\m_R = k$.  Picking a generator $\epsilon$ for $I$ and writing $J = J_0 + \epsilon J_1$ and $J' = J_0 + \epsilon J_1'$, we seek $A \in \gl_r$ such that
\[
 {}^t(\Id + \epsilon A) (J_0 + \epsilon J_1) (\Id + \epsilon A) = J_0 + \epsilon J_1'.
\]
Such an $A$ exists since the map $A \mapsto {}^t A J_0 + J_0 A$ is a surjection from $\gl_r$ to the space of $\epsilon$-symmetric $r$ by $r$ matrices over the field $k$.  The desired result follows by induction, using as base case the fact that $\Mbar$ arose from a representation $\rhobar$ valued in $G(k)$.
\end{proof}

\begin{cor} \label{cor:isdefcond}
$\DFL_\rhobar$ is a deformation condition.
\end{cor}

\begin{proof}
This argument goes back to Ramakrishna~\cite{ram93}, and uses exactness properties of $T_\cris$ on $\MF^{f}_{W,\tor}$, Corollary~\ref{cor:pairingcompat}, and the fact that for a morphism of coefficient rings $R \to R'$, $$R' \tensor{R} T_\cris(M) = T_\cris(R ' \tensor{R} M).$$  For example, to check that $\DFL_\rhobar$ is a sub-functor of $\D_{\rhobar}$, let $R$ be a coefficient ring and $M$ be the Fontaine-Laffaille module corresponding to $\rho \in \DFL_\rhobar(R)$.  Then $R' \tensor{R} T_\cris(M)$ lies in the essential image of $T_\cris$, and $R' \tensor{R} M$ admits a perfect $\epsilon$-symmetric $R'$-bilinear pairing as in Corollary~\ref{cor:pairingcompat} given by extending the pairing on $M$.  This shows that $\rho_{R'} \in \DFL_\rhobar(R')$.  A similar argument checks condition \eqref{defcon2} of being a deformation condition.
\end{proof}

Using Proposition~\ref{fact:tcris}, it is simple to understand characteristic-zero points of the deformation functor.

\begin{prop} \label{prop:charzero}
Suppose we are given a compatible collection of Fontaine-Laffaille deformations $\rho_i : \Gamma_K \to G(R_i)$, where $\{R_i\}$ is a co-final system of artinian quotients of the valuation ring $R$ of a finite extension of $\O[\frac{1}{p}]$ with the same residue field as $\O$.  
Then $\rho = \inverselimit \rho_i$ is crystalline (more precisely, a lattice in a crystalline representation) with indexed tuple of Hodge-Tate weights equal to the corresponding indexed-tuple of Fontaine-Laffaille weights of $\rhobar$.
\end{prop}

\begin{proof}
It is straightforward to verify that the inverse limit of the Fontaine-Laffaille modules corresponding to $\rho_i$ is in $\D_K$.  Then the result follows from combining Fact~\ref{fact:tcris}\eqref{tcris:2} and \eqref{tcris:5}.  Our convention that the cyclotomic character has Hodge-Tate weight $-1$ makes the Hodge-Tate weights and Fontaine-Laffaille weights match (Remark~\ref{rmk:weights}).
\end{proof}

\subsection{Liftability}  In this section, we analyze liftability by constructing lifts of Fontaine-Laffaille modules.  Lifting the underlying module, filtration, and pairing will be relatively easy.  Constructing lifts of the $\varphi^i_M$ compatible with these choices requires substantial work.  
Let $\WFLtau$ denote the Fontaine-Laffaille weights of $\rhobar$ with respect to a $\ZZ_p$-embedding $\tau : W \into \O$, corresponding to the jumps in the filtration of $\Mbar_\tau$.

\begin{prop} \label{prop:flliftable}
 Under the assumption that the Fontaine-Laffaille weights lie in an interval of length $\frac{p-2}{2}$ and are multiplicity-free for each $\tau : W \into \O$, the deformation condition $\DFL_{\rhobar}$ is liftable.
\end{prop}

Let $\rho : \Gamma_K \to G(R)$ be a Fontaine-Laffaille deformation of $\rhobar$.  Let $M$ and $\Mbar$ be the corresponding Fontaine-Laffaille modules for $\rho$ and $\rhobar$, which decompose as 
\[
 M = \bigoplus_\tau M_\tau \quad \text{and} \quad \Mbar = \bigoplus_\tau \Mbar_\tau.
\]
Each $M_\tau$ is a free $R$-module by Lemma~\ref{lem:rfree}.  Furthermore, the filtration $\{M_\tau^i\}$ on $M_\tau$ is given by $R$-module direct summands and $\varphi^i_M(M^i_\tau) \subset M_{\sigma \tau}$.  In particular, there exist free rank-$1$ $R$-modules $N^{w_{\tau,i}}_\tau \subset M_\tau^{w_{\tau,i}}$ such that $M_\tau^{w_{\tau,i} }= N_\tau^{w_{\tau,i}} \oplus M_\tau^{w_{\tau,i}+1}$. 
As the pairing is $\O$-bilinear, the pairings $M_\tau \times M_\tau \to L_\tau$ are collectively equivalent to the pairing $M \times M \to L$, so to lift the pairing and check compatibility it suffices to do so on $M_\tau$.   We also fix a basis for each $L_\tau$, so we may  talk about the value of the pairings.  Thus to analyze liftability of $M$, we will work with each $M_\tau$ separately using $R \tensor{\ZZ_p} W = \prod_{\tau} R_\tau$ with $\tau$ varying through $\ZZ_p$-embeddings $W \into \O \to R$.

By a \emph{basis} for $M_\tau$, we mean a basis for it as an $R$-module.  By Lemma~\ref{lem:rfree}, the rank of $M_\tau$ is $r$. 
For $G = \GSp_r$ with $r$ even, the \emph{standard alternating pairing} with respect to a chosen basis is the one given by the block matrix
\[
 \blockmatrix{0}{I'_{r/2}}{-I'_{r/2}}{0}
\]
where $I'_m$ denotes the anti-diagonal matrix with $1$'s on the diagonal.
For $G = \GO_r$, the \emph{standard symmetric pairing} with respect to the basis is the one given by the matrix $I'_r$.


\begin{example} \label{ex:pairingshape}
Take $R = k$ and fix an embedding $\tau : W \into \O$.  Let $w_1, \ldots , w_r$ be the Fontaine-Laffaille weights of $M_\tau$, and recall that $w_i  + w_{r+1-i} = s_\tau$ because $M \simeq M^\vee$.  Pick $v_i \in M_\tau^{w_i} - M_\tau^{w_i+1}$.  Since $\varphi^i_M |_{M_{i+1}} = p \varphi^{i+1}_M = 0$,
\[
 M_{\sigma \tau} = \sum_{i} \varphi^{i} (M_\tau^i) = \on{span}_k \varphi^{w_i}_M (v_i).
\]
Note that $\{ \varphi^{w_i}_M (v_i)\}$ is a $k$-basis for $M_{\sigma \tau}$, as the left side has $k$-dimension $r$ and there are $r$ Fontaine-Laffaille weights for $\tau$.
Furthermore, compatibility with the pairing means that
\[
 \langle \varphi^{w_i}_M (v_i) , \varphi^{w_j}_M (v_j) \rangle = \varphi^{w_i + w_j}_{L}( \langle v_i, v_j \rangle).
\]
But $\varphi^{h}_L|_{L_\tau} =0$ unless $h = s_\tau$: for $h > s_\tau$ this is because $L_\tau^h =0$, while for $h< s_\tau$ this is because $L^h_\tau = L_\tau^{h+1} = L_\tau$ and $\varphi^{h}_L|_{L_\tau^{h+1}} = p \varphi^{h+1}_L =0$.  Thus $ \langle \varphi^{w_i}_M (v_i) , \varphi^{w_j}_M (v_j) \rangle =0$ unless $w_i + w_j = s_\tau$, in which case the pairing must be non-zero as it is perfect.  If $i \neq j$, by rescaling $v_i$ we may arrange for $\langle \varphi^{w_i}_M (v_i) , \varphi^{w_j}_M (v_j) \rangle$ to be an arbitrary unit.  For $G = \GSp_{r}$ or $G = \GO_{r}$ with $r$ even
this means after rescaling the pairing may be taken to be standard with respect to the basis $n_{i} = \varphi^{w_i}(v_i)$ of $M_{\sigma \tau}$ (and with respect to the fixed basis of $L_\tau$).  For $G = \GO_r$ with $r$ odd and $i = [r/2]+1$, defining $\omega_\tau:= \langle \varphi^{w_i}(v_i), \varphi^{w_i}(v_i) \rangle \in k^\times$ and rescaling $v_1 , \ldots , v_{i-1}$ then brings us to the case that the pairing is $\omega_\tau$ times the standard pairing with respect to the basis 
$n_{i} = \varphi^{w_i}(v_i)$ of $M_{\sigma \tau}$.
\end{example}

\begin{remark}
The constant $\omega_\tau$ depends on the choice of basis $\{v_i\}$ for $M_\tau$, so in particular is not independent of $\tau$.  This will not cause problems in later arguments.
\end{remark}

\begin{remark}
 There is a lot of notation in the following arguments.  With $\tau$ fixed, we will use $v_i$ to denote elements of $M^{w_i}_\tau$, and $m_i$ to denote elements of $M_{\sigma \tau}$.  Usually we will have $\varphi^{w_i}_M(v_i) = m_i$.  If we want to index by Fontaine-Laffaille weights instead of the integers $\{1 ,2,\ldots ,r\}$, we will use $v'_{w_i} := v_i$ and $m'_{w_i} := m_i$.  For a weight $w \in \WFLtau$, let $w^* \in \WFLtau$ denote the unique weight for which $w + w^* = s_\tau$.
\end{remark}

\begin{lem} \label{lem:standardpairing}
Let $w_1 < w_2 < \ldots < w_r$ denote the Fontaine-Laffaille weights of $M$ with respect to $\tau$.  There exists an $R$-basis $m_1 \ldots , m_r$ of $M_{\sigma \tau}$ such that 
$m_i = \varphi^{w_i}_M(v_i)$ where $v_i$ is an $R$-basis for a complement to $M_\tau^{w_i+1}$ in $M_\tau^{w_i}$ and such that the pairing $\langle \cdot , \cdot \rangle$ on $M_{\sigma \tau}$ is an $R^\times$-multiple of the standard pairing with respect to the basis $\{m_i\}$ (and the previously fixed basis of $L$).
\end{lem}

\begin{proof}
Example~\ref{ex:pairingshape} shows that such a basis $\overline{v}_i$ exists over $R/ \m_R$: pick a lift $v_i \in N_\tau^i$ of $\overline{v}_i$, and define $m_i = \varphi^{w_i}_M(v_i)$.  We know that
\[
 \langle \varphi^{w_i}_M v_i , \varphi^{w_j}_M v_j \rangle =  \varphi_L^{w_i+w_j}( \langle v_i, v_j \rangle).
\]
If $w_i + w_j > s_\tau$, this is zero because $L_\tau^{s_\tau + 1 } =0$.  If $w_i + w_j < s_\tau$, since $\varphi_L^{w_i+w_j}|_{L_\tau^{s_\tau}} = p^{s_\tau - w_i -w_j} \varphi_L^{s_\tau}$ this is not a unit.  If $w_i+w_j = s_\tau$ (equivalently, $i+j=r+1$), it is a unit of $R$ as the pairing is perfect.  

We will modify the lifts $v_i$ and then take $m_i = \varphi^{w_i}_M(v_i)$.  For $0 \leq j \leq r/2$ (so $j< r+1-j$), we will inductively arrange that:
\begin{enumerate}
 \item for $i\leq j$, $\langle m_i, m_h \rangle =0$ for $h \neq r+1-i$; \label{induct1}
 \item $v_i$ is an $R$-basis for a complement to $M_\tau^{w_i+1}$ in $M_\tau^{w_i}$; \label{induct2}
 \item  $\langle m_i, m_{r+1-i} \rangle$ is a unit for all $1 \leq i \leq r$. \label{induct3}
\end{enumerate}
For $j=0$, the first condition is vacuous and the other two conditions hold by our choice of lift.  Given that these conditions hold for $j-1$ with $1 \leq j \leq \frac{r}{2}$, we will show how to modify the $v_i$ so that these conditions hold for $j$.  Let $c = \langle m_j,m_{r+1-j}\rangle \in R^\times$.  For $j < h < r+1-j$, define
\[
 \widetilde{v}_h := v_h - \langle m_j , m_h \rangle c^{-1} v_{r+1-j}.
\]
As $j \neq r+1 - h$, $\langle m_j , m_h \rangle \in \m_R$ so $\widetilde{v}_h$ lifts $\overline{v}_h$.  We compute that
 \[
 \langle m_j, \varphi^{w_h}_M \widetilde{v}_h \rangle = \langle m_j, m_h \rangle - \langle m_j ,m_h \rangle c^{-1} \langle m_j, m_{r+1-j} \rangle = 0.
\]
For $i < j$, as $r+1-i \neq h, r+1-h$ we know  $m_i$ is orthogonal to both $m_h$ and $m_{r+1-h}$ by the inductive hypothesis and hence
$\langle m_i , \varphi^{w_h}_M \widetilde{v}_h \rangle = 0$.  Thus \eqref{induct1} holds for the $R$-basis $$v_1, \ldots, v_j, \widetilde{v}_{j+1} , \ldots , \widetilde{v}_{r-j} , v_{r-j+1}, \ldots, v_r.$$

As $\widetilde{v}_h - v_h \in M^{w_{r+1-j}}_\tau$, $\widetilde{v}_h$ is still an $R$-basis for a complement to $M_\tau^{w_h+1}$ in $M_\tau^{w_h}$ (since $w_{r+1-j} > w_h$ as $h < r+1-j$), so \eqref{induct2} holds for this new $R$-basis of $M_\tau$.  Furthermore, we see that 
\[
\langle \varphi^{w_h}_M \widetilde{v}_h ,  \varphi^{w_{r+1-h}}_M \widetilde{v}_{r+1-h} \rangle - \langle m_h , m_{r+1-h} \rangle \in \m_R.
\]
As $\langle m_h , m_{r+1-h} \rangle$ is a unit, $\langle \varphi^{w_h}_M \widetilde{v}_h ,  \varphi^{w_{r+1-h}}_M \widetilde{v}_{r+1-h} \rangle$ is a unit and \eqref{induct3} holds.  Thus we may modify the lifts $v_i$ and then accordingly modify $m_i$ to satisfy the inductive hypothesis.

Take such a basis for $j = [r/2]$.  By \eqref{induct1}, 
\[
 \langle m_i, m_{i'} \rangle =0
\]
if $i+i' \neq r+1$ and one of $i$ or $i'$ is at most $r/2$.  Otherwise $i' > r+1-i$ so $w_i + w_{i'} > s_\tau$ and hence the pairing is zero automatically.  If $r$ is even, rescale $v_1 ,\ldots, v_{r/2}$ so that $\langle m_i, m_{r+1-i} \rangle =1$ for $i \leq r/2$ using \eqref{induct3}.  If $r$ is odd (so $G = \GO_r$), let $\omega_\tau = \langle v_{[r/2]+1}, v_{[r/2]+1} \rangle \in R^\times$ and rescale $v_1, \ldots , v_{[r/2]}$ so that $\langle m_i, m_{r+1-i} \rangle =\omega_\tau$ for $1 \leq i \leq [r/2]$.  In these cases, the pairing with respect to the basis $v_1 ,\ldots , v_r$ is a multiple of the standard pairing.  
\end{proof}

\begin{remark}
When $r$ is odd (so $G= \GO_r$), to choose a basis where the pairing is standard we would need to rescale $v_{[r/2]+1}$ by a square root of the unit $\langle m_{[r/2]+1},m_{[r/2]+1} \rangle$.  This might not exist in $R$.  But note that the orthogonal similitude group $\GO_r$ is unaffected by a unit scaling of the quadratic form.
\end{remark}

Now we begin the proof of Proposition~\ref{prop:flliftable}.  Let $R' \onto R$ be a small surjection with kernel $I$.  To lift $\rho$ to $\rho': \Gamma_K \to G(R')$, we can reduce to the case when $I$ is killed by $\m_{R'}$ and $\dim_k I =1$.  Lift the $R$-module $M_\tau$ together with its pairing $\langle \cdot , \cdot \rangle$ over $R'$ as follows.  Choose the basis $\{m_i\}$ provided by Lemma~\ref{lem:standardpairing}, with respect to which $\langle \cdot , \cdot \rangle$ is $\omega_\tau$ times the standard pairing for some $\omega_\tau \in R^\times$.  We take $M'_{\sigma \tau}$ to be a free $R'$-module with basis $\{n_i\}$ reducing to the basis $\{m_i\}$ of $M_{\sigma \tau}$.  Lift $\omega_\tau$ to some $\omega'_\tau \in (R')^\times$ and define a pairing on $M'_\tau$ to be $\omega'_\tau$ times the standard pairing on $M'_\tau$ with respect to $\{n_i\}$.  
Pick a lift $u_i \in M'_\tau$ of $v_i$, and define a filtration on $M'_\tau$ by
\[
 (M'_\tau)^j = \on{span}_{R'}( u_i :  w_i \geq j).
\]
We define the module $M' = \bigoplus_{\tau : W \into \O} M'_\tau$ over $W \tensor{\ZZ_p} R$ with filtration $(M')^i = \bigoplus _{\tau : W \into \O} (M'_\tau)^i$.  It is clear the filtration reduces to the filtration on $M$.  Furthermore, the pairing $M'_\tau \times M'_\tau \to L_{\tau}$ with respect to $\{n_i\}$ is a multiple of the standard one.

It remains to produce $\varphi_{M'}^i$ lifting $\varphi_M^i$.  As always, it suffices to lift all of the $\varphi_{M_\tau}^i : M_\tau^i \to M_{\sigma \tau}$ separately.  We note that the $\varphi_{M'_\tau}^{j} : M'^{j}_\tau \to M'_{\sigma \tau}$ are determined by the values $\varphi_{M'_\tau}^{w_i}(u_i)$ for $w_i \in \WFLtau$ and the relation $p \varphi_{M'_\tau}^{j+1} = \varphi_{M'_\tau} ^{j}|_{M'^{w_j+1}_\tau}$.
We will define $\varphi_{M'_\tau}^{w_i}(u_i)$ for each $w_i \in \WFLtau$ to obtain the desired set of maps  $\varphi_{M'}^j : M'^j \to M'$.

It will now be more convenient to index via weights, so let $n'_{w_i} = n_i$ and $u'_{w_i} = u_i$.  Let us consider defining 
\[
 \varphi_{M'_\tau}^{w}(u'_w) = \sum_{i \in \WFLsigma} c_{i w} n'_i := x_w
\]
for $c_{i w}$ to be determined with the obvious restriction that $c_{i w}$ must lift the corresponding coefficient for $\varphi_M^w(v'_w)$.  We will study for which choices of $\{ c_{i w} \}$ these maps are compatible with the pairing.

\begin{lem} \label{lem:basis}
For any choice of $\{ c_{i w} \}$, the elements $x_w$ form a basis for $M'_{\sigma \tau}$.
\end{lem}

\begin{proof}
Note that the Fontaine-Laffaille weights of $\Mbar$, $M$, and ${M'}$ are the same.
Consider the map $$\sum_{i \in \WFLtau} \varphi_{M'_\tau}^i : M'^i_\tau \to M'_{\sigma \tau}.$$  Quotienting by the maximal ideal of $R'$, as $\varphi_M'^w$ is a lift of $\varphi_{\Mbar}^w$ we obtain a surjection
\[
 \sum_{i \in \WFLtau} \varphi_{\Mbar}^i : \Mbar^i_\tau \onto \Mbar_{\sigma \tau}
\]
as  $\Mbar_{\sigma \tau} = \sum_i \varphi_{\Mbar_\tau}^i(\Mbar^i_\tau)$.  By Nakayama's lemma, the original map is also a surjection.  Thus $\{x_w\}$ spans the free $R$-module $M'_{\sigma \tau}$.  But $\# \{ x_w\} = \on{rk}_{R'} ({M'}_{\sigma \tau})=r$, so $\{x_w\}$ is a basis for $M'_{\sigma \tau}$.
\end{proof}

The compatibility condition with the pairing is that
\[
 \langle \varphi^i_{M'_\tau}(x), \varphi^j_{M'_\tau}(y) \rangle = \varphi^{i+j}_{L_{\tau}} \left( \langle x,y\rangle \right).
\]
Let $\epsilon =1$ for $\GO_r$ and $\epsilon=-1$ for $\GSp_r$ with even $r$.  For a Fontaine-Laffaille weight $i \in \WFLtau$, $n'_i$ and $n'_{i^*}$ pair non-trivially as $i + i^* = s_\tau$.  By linearity and the relation $\langle x,y \rangle = \epsilon \langle y ,x \rangle$, it suffices to check compatibility with the pairing only when $i, j \in \WFLtau$, $x = n'_i$ and $y= n'_j$ and $i <j$ or $i=j=i^*$ (provided we have arranged that $p \varphi_{M'}^{w+1} = \varphi^w_{M'} |_{{M'}^{w+1}}$).

\begin{remark} \label{rmk:orthogonalcase}
 The case $i=j=i^*$ only occurs when the pairing is orthogonal and $r$ is odd, for the weight of the unique basis vector which pairs with itself giving a unit.
\end{remark}

Of course, there is no reason to expect our initial arbitrary choice of $\{ c_{i w} \}$ to work.  Any other choice is of the form $\{ c_{i w}  + \delta_{i w}\}$ where $\delta_{i,w} \in I$.  The compatibility condition on $M'_\tau$ becomes
\[
 \sum_{w,w' \in \WFLtau} (c_{i w} + \delta_{i w}) (c_{j w'} + \delta_{j w'}) \langle n'_w, n'_{w'} \rangle = \varphi^{i+j}_{L_{\tau}} \left( \langle n'_i, n'_j \rangle \right).
\]
Expanding and using the fact that $I^2 = 0$, we see that we wish to choose $\{\delta_{i w}\}$ so that
\[
 \sum_{w,w' \in \WFLtau} \left( c_{i w} \delta_{j w'} +  c_{j w'} \delta_{i w}\right) \langle n'_w, n'_{w'} \rangle =  \omega'_\tau C_{ij}  
\]
where the constant $C_{ij} := (\omega'_\tau)^{-1} \left( \varphi_{L_\tau}^{i+j}(n'_i,n'_j) - \sum_{w, w' \in \WFLtau} c_{i w} c_{j w'} \langle n'_w, n'_{w'} \rangle \right) $ lies in $I$ as $\varphi^i_M$ is compatible with the pairing.

Now we can simplify based on the explicit form of the pairing with respect to the basis $\{ n'_w \}$.  As $n'_w$ only pairs non-trivially with $n'_{w^*}$, we obtain the relation (for $i<j$ or $i =j = i^*$)  
\begin{align} \label{eq:relation}
 \sum_{w \leq w^*} \left( c_{i w} \delta_{j w^*} +  c_{j w ^*} \delta_{i w} \right) + \epsilon \sum_{w> w^*} \left( c_{iw } \delta_{jw^*} +  c_{jw^*} \delta_{i w} \right)  =  C_{ij}.
\end{align}
To show that this system of linear equations has a solution, we shall interpret it as a linear transformation.

It is now convenient to index the weights using $\{1,2,3\ldots ,r\}$.  Recall that the Fontaine-Laffaille weights of $M_\tau$ are denoted $w_1 < w_2 < \ldots < w_r$.  Let $U = I^{\oplus r^2}$, and decompose $U$ as $\bigoplus_{i=1}^r U_{i}$, where the coordinates of $U_{i} = I^{\oplus r}$ are denoted$\{\delta_{w_i,w_j}\}_{j=1}^r$.  Let $U' = I^{\oplus \frac{r(r-1)}{2} + \sigma_r}$, where $\sigma_r = 1$ if there is a $w \in \WFLtau$ for which $w=w^*$ and $0$ otherwise.  (So $\sigma_r$ is zero unless $G =\GO_r$ and $r$ is odd.)  We may write $U' = \bigoplus_{i=1}^{r-1} U'_i$, where the coordinates of $U'_i = I^{\oplus r-i}$ are denoted $\{C_{w_i w_j}\}_{j=i+1}^{r}$, except if $\sigma_r=1$ and $w_i= w_i^*$.  In that case, instead take $U_i' = I^{\oplus r-i+1}$ with coordinates denoted $\{C_{w_i w_j}\}_{j=i}^{r}$.

Consider the function $T : U \to U'$ given by
\begin{align*}
 (\delta_{w_i w_h})_{i h} \mapsto \left(C_{w_i w_j}= \sum_{w_h \leq w_h^*} \left( c_{w_i w_h } \delta_{w_j w^*_h} +  c_{w_j w^*_h} \delta_{w_i w_h} \right) + \epsilon \sum_{w_h> w^*_h} \left( c_{w_i w_h} \delta_{w_j w^*_h} +  c_{w_j w^*_h} \delta_{w_i w_h} \right)  \right)_{ij}
\end{align*}
where the $c_{w w'} \in R'$ matter only through their images in $k$ since $\m_{R'} I = 0$.  
It suffices to show that $T$ is surjective.  As we arranged for $I$ to be $1$-dimensional over $R'/\m_{R'} = k$, this is a question of linear algebra over $k$ upon fixing a $k$-basis of $I$.

We will study particular $k$-linear maps $U_i \to U'_i$.  To simplify notation, let $\epsilon_i = 1$ except when $w_i > w_i^*$ and the pairing is alternating ($\epsilon = -1$), in which case $\epsilon_i = -1$.

\begin{lem} \label{lem:ti}
Suppose $w_i \neq w_i^*$.
The linear transformation $T_i : U_i \to U'_i$ defined on 
\[
(\delta_{w_i w_h})_h \mapsto  \left(C_{w_i w_j} = \sum_{h=1}^r \epsilon_h c_{w_j w^*_h} \delta_{w_i w_h}\right)_j
\]
is surjective.  It is the composition $U_i \to U \overset{T} \to U' \to U'_i$.
\end{lem}

\begin{proof}
As $I$ is one-dimensional over $R/\m_R =k$, it suffices to study the matrix for this linear transformation with respect to a fixed $k$-basis of $I$.  Fix $w_{h'} \in \WFLtau$.  If we take $\delta_{w_i w_h}=0$ for $w_h \neq w_{h'}$ and $\delta_{w _i w_{h'}} =1$, the image of $\{\delta_{w_i w_h} \}_h \in U_i$ under $T_i$ has coordinates $C_{w_i w_j} = \epsilon_{w_{h'}} c_{w_j w^*_{h'}}$.  Thus the matrix for $T_i$ is
\[
 \begin{pmatrix}
 \epsilon_1 c_{w_{i+1} w_1^*} & \epsilon_2 c_{w_{i+1} w_2^*} & \ldots & \epsilon_r c_{w_{i+1} w_r^*}\\
  \epsilon_1 c_{w_{i+2} w_1^*} & \epsilon_2 c_{w_{i+2} w_2^*} & \ldots & \epsilon_r c_{w_{i+1} w_r^*} \\
  \ldots & \ldots & \ldots & \ldots  \\
  \epsilon_1 c_{w_r w_1^*} & \epsilon_2 c_{w_r w_2^*} & \ldots & \epsilon_r c_{w_r w_r^*}
 \end{pmatrix}.
\]
Multiplying the $i$th column by $\epsilon_i$, the columns of this matrix are exactly the coordinates of $x_{w_j}$ with respect to the basis $\{n'_w\}_{w \in \WFLsigma}$ as in Lemma~\ref{lem:basis} except that the first $i$ rows are removed.  As the $\{x_w\}$ form a basis, the columns of this matrix span $U'_i$.

The last statement follows from the definition.
\end{proof}

\begin{remark}
The statement for $w_i=w_i^*$ is similar.  In that case, we must have $\epsilon =1$, and we have
\[
 C_{w_i w_i} = 2 \sum_{j} c_{w_i w_j^*} \delta_{w_i w_j}.
\]
Extending the definition of $T_i$ in Lemma~\ref{lem:ti}, we again see that the columns of the matrix representing this transformation are truncated versions of the coordinates of $x_{w_j}$ with some signs changed and one coordinate multiplied by $2$.  The image of a basis under the transformation multiplying one coordinate by $2$ is still a basis, so again $T_i$ is surjective.
\end{remark}

\begin{lem}
The composition $T_{i j} : U_i \to U \overset{T} \to U' \to U'_{j}$ is zero whenever $i< j$.
\end{lem}

Informally, this is saying that $T$ is block lower-triangular with diagonal blocks that are surjective.

\begin{proof}
The coordinates of $U_{i}$ are $\delta_{w_i w_h}$.  The coordinates of $U'_j$ are $C_{w_j w_h}$ for $j < h$ (or $j \leq h$ if $w_j = w_j^*$).  Looking at the formulas for $C_{w_j w_h}$ in the definition of $T$, they depend only on certain $\delta_{w w'}$ with $w \neq w_i$: this uses that $i<j \leq h$ to rule out any $\delta_{w_i w'}$ from appearing.  These are all zero on the image of the inclusion $U_i \to U$, so the composition is zero.
\end{proof}

\begin{cor} \label{cor:surjective}
$T$ is surjective.
\end{cor}

\begin{proof}
The composition of $U_i \to U \to U' \to U'_i$ is exactly $T_i$, hence surjective.  For $v \in U'$, by descending induction on $i$, we will construct $u_i \in U_{i}$ so that
\[
 T(u_i + \ldots + u_r) - v \in U'_{1} \oplus \ldots \oplus U'_{{i-1}}
\]
(meaning $T(u_1 + \ldots + u_r) = v$ when $i=1$).  
For $i=r$, take $u_r$ be a preimage under $T_r$ of the component of $v$ in $U'_r$.  Now suppose we have selected $u_{i+1} , \ldots u_r$.  Pick a preimage $u_i \in U_i$ of the projection of $T(u_{i+1} + \ldots u_r) - v$ to $U'_i$ using the surjectivity of $T_i$.  We know that $T_{i j}(u_i) = 0$ for $j>i$, so 
\[
 T(u_i + \ldots + u_r) - v \in  U_{1}' \oplus \ldots \oplus U'_{i-1}.
\]
For $i=1$, we have $T(u_1 + \ldots + u_r) = v$ as desired.
\end{proof}

Corollary~\ref{cor:surjective} lets us choose the $\{\delta_{i h}\}$ so that the compatibility relations \eqref{eq:relation} are satisfied.  This defines $\varphi^w_{M'_\tau}(n'_w)$, and hence we can extend to a map $\varphi^i_{{M'}} : {M'}  \to {M'}$ compatible with the pairing.  
We then finish the proof of Proposition~\ref{prop:flliftable} as follows.

 Given the deformation $\rho$ to a coefficient ring $R$ with associated Fontaine-Laffaille module $$M = \bigoplus_{\tau : W \into \O} M_\tau,$$ and a small surjection $R' \to R$ whose kernel $I$ is $1$-dimensional over the field $R'/ \m_{R'}$, we have constructed a free $R'$-module ${M'}$ together with a filtration $\{(M')^i\}$ and maps $\varphi^i_{M'}$ by lifting the $M_\tau$.  The filtration and $\{\varphi_{M'}^i\}$ make ${M'}$ into a Fontaine-Laffaille module.  There is an obvious $R' \tensor{\ZZ_p} W$-module structure.  The condition $M' = \sum_{i} \varphi_{M'}^i({M'}^i)$ follows from Lemma~\ref{lem:basis}.  We also constructed a pairing ${M'} \times {M'} \to L$, and the filtration and $\varphi^i_{M'}$ are compatible with it (in the sense of Corollary~\ref{cor:pairingcompat}) by our choice of $(\delta_{ih})_{ih}$.  By Corollary~\ref{cor:pairingcompat} and Lemma~\ref{lem:rfree}, $T_\cris({M'})$ gives a representation $\rho' : \Gamma_K \to G(R')$ lifting $\rho$.

\subsection{Tangent Space}

The final step in the proof of Theorem~\ref{thm:fl} is to analyze the tangent space of $\DFL_\rhobar$.  It is a subspace $\LFL_\rhobar$ of the tangent space $H^1(\Gamma_K,\adzerorho)$ of deformations with fixed similitude character $\nu$.  We are mainly interested in its dimension as a vector space over $k$, and will analyze it by considering deformations $\rho$ of $\rhobar$ to the dual numbers $k[t]/(t^2)$.  Recall that $G = \GSp_{r}$ (with even $r$) or $G = \GO_r$; let $B$ be a Borel subgroup of $G$.

\begin{prop} \label{prop:fltangent}
Under the standing assumption that $\rhobar$ is torsion-crystalline with pairwise distinct Fontaine-Laffaille weights for each $\tau : W \into \O$ contained in an interval of length $\frac{p-2}{2}$, 
\[
 \dim_k \LFL_\rhobar - \dim_k H^0(\Gamma_K,\adzerorho) = [K:\QQ_p] (\dim G_k - \dim B_k).
\]
\end{prop}

Let $\Vbar$ be the Galois module given by $\rhobar$, and for a lift $\rho$ of $\rhobar$ to $k[t]/(t^2)$ let $V$ be the corresponding Galois module.  The submodule $t V$ is naturally isomorphic  to $\Vbar$, and we have an exact sequence
\[
 0 \to t V \to V \to \Vbar \to 0.
\]
Let $\Mbar$ be the Fontaine-Laffaille module corresponding to $\rhobar$, with pairing $\langle \cdot , \cdot \rangle : \Mbar \times \Mbar \to L_k$.  We know $\Mbar$ is a $k$-vector space of dimension $r [K:\QQ_p]$.  Let $M$ be the Fontaine-Laffaille module corresponding to $\rho$.  It is a free $k[t]/(t^2)$-module, and fits in an exact sequence
\[
 0 \to t M \to M \to \Mbar \to 0
\]
of Fontaine-Laffaille modules.  The map $\overline{M} \subset M \to t M$ induced by multiplication by $t$ is an isomorphism of Fontaine-Laffaille modules since it is so on underlying $k$-vector spaces using the $k[t]/(t^2)$-freeness of $M$.  As before, we have decompositions
\[
 M = \bigoplus_{\tau : W \into \O} M_\tau \quad \text{and} \quad \Mbar = \bigoplus_{\tau : W \into \O} \Mbar_\tau
\]
from Lemma~\ref{lem:ostructure}.

Using Lemma~\ref{lem:standardpairing}, pick a basis $\{v_{\tau,i}\}_{i=1}^r$ of the $k[t]/(t^2)$-module $M_\tau$ such that $v_{\tau,i}$ is a basis for a $k[t]/(t^2)$-complement to $M_{\tau}^{w_i+1}$ in $M_\tau^{w_i}$ and such that the pairing $M_{\sigma \tau} \times M_{\sigma \tau} \to L_{\sigma \tau}$ with respect to the $m_{\tau,i} := \varphi^{w_i}_M (v_{\tau,i})$ is $\omega_\tau$-times the standard pairing.  As $1$-units admit square roots, we may assume that $\omega_\tau \in k^\times$.  Note that $\{m_{\tau,i}\} \cup \{ t m_{\tau,i} \}$ is a basis for $M_{\sigma \tau}$ as a $k$-vector space, and $\{ m_{\tau,i} \}_{\tau,i}$ is a basis for $M$ as a $k[t]/(t^2)$ module.

Let $M_0$ be the subspace of $M$ spanned by the $\{v_{\tau,i}\}_{\tau,i}$ as a $k$-vector space.  We have that $t M_0 = t M \simeq \Mbar$ as vector spaces, and have an obvious decomposition
\[
 M_0 = \bigoplus_{\tau : W \into \O} M_{\tau,0}.
\]
We obtain a pairing on $M_0$ by restriction and a filtration by intersection: $M_{\tau,0}^i = M^i \cap M_{\tau,0}$.  


\begin{lem}
We have that $M_\tau^i = M_{\tau,0}^i \otimes k[t]/(t^2)$, and hence $M^i = M_0^i \otimes k[t]/(t^2)$.
\end{lem}

\begin{proof}
We know that the $k[t]/(t^2)$-span of $v_i$ is a $k[t]/(t^2)$-complement to $M_\tau^{w_i+1}$ in $M_\tau^{w_i}$.  Hence $M_\tau^{w_i}  / M_\tau^{w_i+1}$ is isomorphic to the $k$-span of $v_i$ and $t v_i$.  
As the filtration is automatically split ($M^i_\tau$ is a direct summand of $M_\tau$, and hence $M^i_\tau$ is a direct summand of $M^{i-1}_\tau$), this suffices.
\end{proof}

Observe that the surjection of Fontaine-Laffaille modules $M \to \Mbar$ carries $M_0$ isomorphically onto $\Mbar$.  Under the isomorphism of $k$-vector spaces $M_0 \to \Mbar$, the pairing on $M_0$ and the pairing on $\Mbar$ are identified because by choice of basis the pairing on $M_0$ is a $k^\times$-multiple of the standard pairing.  Furthermore, extending the pairing $M_0 \times M_0 \to L$ by $k[t]/(t^2)$-bilinearity recovers the pairing on $M$.  Using $M_0 \simeq \Mbar$, we can also define $\varphi_{M_0}^i : M_0^i \to M_0$ to be the lift of $\varphi_{\Mbar}^i$ to $M_0^i$.  It is compatible with the pairing on $M_0$.  Note that it is \emph{not} the same as $\varphi_M^i|_{M_0^i}$.

Our goal is to describe the set of strict equivalence classes of deformations $M$ of $\Mbar$, so by making these identifications it remains to study ways to lift $\varphi_{\Mbar}^i$ to a map $$\varphi^i_{M_0 \otimes k[t]/(t^2)} : M_0^i \otimes k[t]/(t^2) \to M_{0} \otimes k[t]/(t^2).$$  For $n, n' \in M_0^i$ we may write 
\[
 \varphi^i_M(n + t n') = \varphi^i_{M_0}(n) + t (\varphi^i_{M_0}(n') + \delta_i(n))
\]
for some $\sigma$-semilinear $\delta_i : M_0^i \to M_0$ which completely determines $\varphi^i_M$.  It is clear that for $n \in M_0^{i+1}$ we have $\delta_i(n)=0$ due to the relation $\varphi^i_{M_0}(n) = p \varphi^{i+1}_{M_0}(n) =0$.  Thus, $\delta_i$ factors through $M_0^i/M_0^{i+1}$, and together the $\delta_i$ define a $\sigma$-semilinear
\[
 \delta : \gr^\bullet(M_0) \to M_0.
\]
Compatibility with the pairing says exactly that
\[
 \langle \varphi^i_M(n + tn'), \varphi^j_M(m+tm') \rangle = \varphi_L^{i+j}( \langle n+tn',m+tm' \rangle )
\]
for $n, n'\in M_0^i$ and $m , m' \in M_0^j$ and all $i$ and $j$.  
Expanding and using the compatibility of the $\varphi_{M_0}^i$ with the pairing, we see that it is necessary and sufficient that 
\begin{equation} \label{eq:flpairingcompat}
\langle \delta_i(n),\varphi^j_{M_0}(m) \rangle + \langle \varphi^i_{M_0}(n), \delta_j(m) \rangle = 0
\end{equation}
for $n \in M_0^i$ and $m \in M_0^j$ and all $i$ and $j$.  
As $\Mbar = \sum_i \varphi^i_{\Mbar} (\Mbar^i)$ and we defined $\varphi^i_{M_0}$ to lift $\varphi^i_{\Mbar}$, it follows that $M_0 = \sum_i \varphi^i_{M_0}(M_0^i)$.  Furthermore, we have an isomorphism $\varphi : \gr^\bullet (M_0) \to M_0$.  This allows us to rewrite \eqref{eq:flpairingcompat} as the requirement that for $m,n \in \gr^\bullet (M_0)$,
\[
  \langle \delta' \varphi(n), \varphi(m) \rangle + \langle \varphi(n), \delta' \varphi(m) \rangle = 0
\]
where $\delta'$ is the $k$-linear composition of $\varphi^{-1}$ with $\delta$.  
In other words, 
\[
 \langle \delta' x, y \rangle + \langle x, \delta' y \rangle =0
\]
for all $x,y \in M_0$.  Note that $\delta'$ is compatible with the filtration, the pairing, and the $k \otimes W$-module structure.   Denote the collection of all such $\delta'$ by $\End_{k \otimes W}(M_0,\langle \cdot, \cdot \rangle)$: it is isomorphic to $\mathfrak{sp}_{r}(k \otimes W)$ or $\mathfrak{so}_r(k \otimes W)$, which have dimension $[K:\QQ_p] (\dim G_k -1)$ over $k$.

\begin{lem}
For such a choice of $\delta'$, we obtain a Fontaine-Laffaille module $M \in \MF_{W,\tor}^f$ together with a pairing $M \times M \to L$ as in Corollary~\ref{cor:pairingcompat}.
\end{lem}

\begin{proof}
This is just bookkeeping.  First, observe that $\sum_i \varphi^i_M (M^i)$ is a $k[t]/(t^2)$-module containing $\varphi^i_{M_0}(M_0^i) = M_0$.  Thus it is $M$.  It is immediate that the pairing is compatible with the filtration.  We chose $\delta'$ so that the pairing is compatible with the $\varphi^i_M$.
\end{proof}

Of course, different $\delta'$ may give isomorphic deformations of $\Mbar$.  Suppose that we are given $\delta$ and $\gamma$ such that the Fontaine-Laffaille modules they create are strictly equivalent as deformations of $\Mbar$ (in the sense that they are isomorphic Fontaine-Laffaille modules and their reductions are identified with $\Mbar$ compatibly with the isomorphism, or equivalently that they give the same element of $\DFL_\rhobar(k[t]/(t^2))$).  We have shown that the underlying module, pairing, and filtration can be identified with the fixed data $M = M_0 \otimes k[t]/(t^2)$, $\langle \cdot, \cdot \rangle \otimes k[t]/(t^2)$, and $M_0^i \otimes k[t]/(t^2)$.  The isomorphism reduces to the identity modulo $t$ (by strictness).  This means there exists an isomorphism $\alpha : M_0 \to M_0$ compatible with the pairing, filtration, and module structure such that 
\[
 (1 + t \alpha) \left( \varphi^i_{M_0}(n) + t (\varphi^i_{M_0}(n') + \delta_i(n)) \right) =\varphi^i_{M_0}(n) + t (\varphi^i_{M_0}(\alpha(n) + n') + \gamma_i(n)) .
\]
Simplifying, this is the condition that
\[
 \gamma_i(n) - \delta_i(n) = \alpha( \varphi^i_{M_0}(n)) - \varphi^i_{M_0}(\alpha(n)).
\]
In other words, $\delta, \gamma \in \End(M_0,\langle \cdot, \cdot \rangle)$ define the same deformation if and only if $\gamma_i - \delta_i$ is of the form $\alpha \circ \varphi^i_{M_0} - \varphi^i_{M_0} \circ \alpha$ for all $i$ and some $\alpha \in \Fil^0 \End_{k \otimes W}(M_0,\langle \cdot, \cdot \rangle)$.  This means that $\alpha$ is a $k \otimes W$-linear endomorphism of $M_0$ that is compatible with the filtration and pairing.  We can identify $\End_{k \otimes W}(M_0,\langle \cdot, \cdot \rangle)$ with the Lie algebra of a symplectic or orthogonal group valued in $k \otimes W$.  The filtration defines a Borel subgroup of this symplectic or orthogonal group, whose Lie algebra is $\Fil^0 \End_{k \otimes W}(M_0,\langle \cdot, \cdot \rangle)$.  (The assumption that the Fontaine-Laffaille weights for each $\tau$ are pairwise distinct is what makes it a Borel subgroup.)    Hence the dimension of $\Fil^0 \End_{k \otimes W}(M_0,\langle \cdot, \cdot \rangle)$ as a $k \otimes W$-module space is the dimension of this Borel in the symplectic orthogonal group.  We conclude that
\[
 \dim_k \Fil^0 \End_{k \otimes W}(M_0,\langle \cdot, \cdot \rangle) = [K:\QQ_p] (\dim B_k -1)
\]
where $B$ is a Borel in the symplectic or orthogonal similitude group. 

Finally, we must understand when $\alpha$ and $\beta$ satisfy
\[
 \alpha \circ \varphi^i_{M_0} - \varphi^i_{M_0} \circ \alpha = \beta \circ \varphi^i_{M_0} - \varphi^i_{M_0} \circ \beta.
\]
This happens exactly when $\alpha - \beta$ commutes with the $\varphi^i_{M_0}$ (as well as being compatible with the filtration, pairing, and module structure).  In other words, $\alpha - \beta \in \End_{\MF_W}(M_0,\langle \cdot, \cdot \rangle)$.  But under $T_\cris$, this is identified with endomorphisms of $\rhobar$ preserving the pairing (not just up to a similitude factor), and in particular has dimension $\dim_k H^0(\Gamma_K,\adzerorho)$.

We can express this analysis as the exact sequence
\[
 0 \to \End_{\MF_W}(M_0,\langle \cdot, \cdot \rangle) \to \Fil^0\left (\End_{k \otimes W}(M_0,\langle \cdot, \cdot \rangle) \right)\to \End_{k \otimes W}(M_0,\langle \cdot, \cdot \rangle) \to \DFL_\rhobar(k[t]/(t^2)) \to 0.
\]
We finish the proof of Proposition~\ref{prop:fltangent} by taking dimensions:
\begin{align*}
\dim_k \LFL_\rhobar - \dim_k H^0(\Gamma_K,\adzerorho) &= [K:\QQ_p] (\dim G_k -1)- [K:\QQ_p](\dim B_k -1) \\
&= [K:\QQ_p] (\dim G_k - \dim B_k). 
\end{align*}

\appendix

\section{Some Details about Fontaine-Laffaille Modules}

Since Conrad's notes \cite{conrad94} are not publicly available, we have extracted his arguments justifying some facts for which we do not know another reference in the literature.  Any errors are of course my own.  We continue the notation of \S\ref{sec:covariantfl}. 

\subsection{Limits}   For $M \in \DK$, there is a natural map of $\ZZ_p[\Gamma_K]$-modules
\[
 T_{\cris}(M) \to \inverselimit_n T_{\cris} (M / p^n M).
\]
We will prove it is an isomorphism.  We can twist to reduce to the case where $M^{2-p}=M$ and $M^{1} =0$ and so 
$$T_{\cris}(M) = \ker \left(1 - \varphi^0_{\Acris \otimes M} : \Fil^0(\Acris \otimes M) \to \Acris \otimes M \right).$$
The key is that the natural map
\begin{equation} \label{eq:**}
 \Acris \otimes M \to \inverselimit_n (\Acris \otimes M/ p^n M) 
\end{equation}
is an isomorphism which respects the filtration, Galois action, and Frobenius.  It is a bijection since $M$ is $p$-adically complete (being finitely generated over $W$) and $\Acris$ is $p$-adically complete by definition.  It is is simple to see it respects the Galois action and Frobenius. The claim about the filtration is the interesting one, and follows from the fact that the filtered pieces of $M$ are direct summands.  Since \eqref{eq:**} is compatible with the Frobenius and filtration, and inverse limits are left exact, the kernel of $1 - \varphi^0$ on the left is the limit of the kernels of $1 - \varphi^0$ on $A_\cris \otimes M/p^n M$.

\subsection{Construction of Duals}  
In Lemma~\ref{lem:mfdual}, we needed to construct maps $\varphi^i_{M^*}$.  We have already adapted that argument in Lemmas~\ref{lem:flduality1} and \ref{lem:flduality}, so do not repeat it here.

\subsection{Tensor Products}  We now sketch the argument for Fact~\ref{fact:tensor}.  After twisting, we may assume we Fontaine-Laffaille module $M$ and $N$ such that $M, N, M \otimes N \in \MF_{W,\tor}^{f,[2-p,1]}$.  There is a natural map $T_\cris(M) \otimes T_\cris(N) \to T_\cris(M \otimes N)$ given by the map
\begin{equation} \label{eq:tensoriso}
\ker \left(1 - \varphi^0_{\Acris \otimes M} \right) \otimes \ker \left(1 - \varphi^0_{\Acris \otimes N}  \right) \to \ker \left(1 - \varphi^0_{\Acris \otimes M \otimes N}\right)
\end{equation}
induced by the multiplication map $A_\cris \otimes A_\cris \to A_\cris$.  To check it is an isomorphism, we check that it is an isomorphism when $M$ and $N$ are simple objects and the residue field $k$ is algebraically closed, and then use a d\'{e}visagge argument.  

Let us first sketch the d\'{e}vissage to the case of algebraically closed residue field and simple objects.  Extending the field is exact, so passing to the algebraic closure is no problem.  The only subtle point in the rest of the argument is that tensor products are only right exact.  Given $M_1 \into M$ with quotient $M_2$, we obtain the diagram
\[
\xymatrix{
T_\cris(M_1) \otimes T_\cris(N) \ar[r] \ar[d] &  T_\cris(M) \otimes T_\cris(N) \ar[r] \ar[d]&  T_\cris(M_2) \otimes T_\cris(N) \ar[r] \ar[d] & 0 \\
 T_\cris(M_1 \otimes N) \ar[r] & T_\cris(M \otimes N) \ar[r] & T_\cris(M_2 \otimes N) \ar[r]  & 0.
} 
\]
A diagram chase shows the middle vertical map is surjective if the left and right maps are isomorphisms.  But since there is a non-canonical isomorphism $M \simeq \O_K \tensor{\ZZ_p} T_{\cris}(M)$ \cite[Remarque 3.4]{fl82}, the middle map is automatically an isomorphism as the middle terms have the same length.

Checking that \eqref{eq:tensoriso} is an isomorphism when the residue field $k$ is algebraically closed and $M$ and $N$ are simple is more involved.
 We can use duality as in Remark~\ref{remark:tcrisstar} to instead check that the dual map
\begin{equation} \label{eq:tocheck}
 T_\cris^*(M) \tensor{k} T_\cris^*(N) \to T_\cris^*(M \otimes N)
\end{equation}
given by multiplication on $A_{\cris}/p A_\cris$ is an isomorphism ($M$ and $N$ are $p$-torsion as they are simple).  
Recall that the simple objects are $M(h;i)$, where $h$ is a positive integer and $i : \ZZ/ h \ZZ \to \ZZ$ is a function with minimal period $h$ and value at $n$ denoted by $i_n$.  It is defined by 
\[
 M(h;i) = \bigoplus_{n \in \ZZ/h\ZZ} k e_n, \quad M(h;i)^j = \bigoplus_{i_n \geq j} k e_n, \quad \varphi^{i_n}_{M(h;i)} (e_n) = e_{n-1}.
\]
We write $M = M(h;i)$ and $N = M(h';i')$.
As a $k$-vector space, $M(h;i) \otimes M(h';i')$ has basis $e_n \otimes e'_m$ for $1 \leq n \leq h$ and $1 \leq m \leq h'$.  We see that $e_n \otimes e_m' \in (M(h;i) \otimes M(h';i'))^j$ if and only if $i_n + i'_m \geq j$, and that $\varphi^{i_n + i'_m}_{M(h;i) \otimes M(h';i')} ( e_n \otimes e_m') = e_{n-1} \otimes e'_{m-1}$.  Thus we may decompose
\begin{align*}
M(h;i)\otimes M(h';i') & \simeq \bigoplus_{s=0}^{\on{gcd}(h,h')-1} \bigoplus_{r \in \ZZ/ \on{lcm}(h,h') \ZZ} k  \cdot e_r \otimes e'_{r+s} \\
& \simeq \bigoplus_{s=0}^{\on{gcd}(h,h')-1} M( \on{lcm}(h,h') ; r \mapsto i_r + i'_{s+r} ).
\end{align*}
This respects the Fontaine-Laffaille module structure, and has Fontaine-Laffaille weights between $0$ and $p-2$ if and only if $\max i_n + \max i'_m < p-2$.

For fixed $0\leq s < \gcd(h,h')$, let $h_s$ denote the minimal period of the function $i'' : r \mapsto i_r + i'_{s+r}$.  By \cite[Remarque 4.11]{fl82}, $M(\on{lcm}(h,h');i'')$ is a direct sum of $d_s=\on{lcm}(h,h') / h_s$ copies of $M(h_s,i'')$.  For each simple factor, we obtain an embedding 
\[
 M(h_s; i'') \into M(\on{lcm}(h,h'); i'') \into M(h;i) \otimes M(h';i').
\]
Applying $T_{\cris}^*$ and composing with the map \eqref{eq:tocheck}, we obtain a map
\[
 T_\cris^*(M(h;i)) \otimes T_\cris^*(M(h';i')) \to T_\cris^*(M(h_s;i'')).
\]
If this map is non-zero for all choices of $s$ and all embeddings, then by simplicity and dimension counting \eqref{eq:tocheck} is an isomorphism.

The argument toward the end of \cite[Lemma 4.9]{fl82} adapts to give an explicit embedding of $M(h_s;i'')$ in terms of the bases for $M(h;i)$, $M(h';i')$, and $M(h_s;i'')$, yielding an explicit $$\alpha_{s} : T_\cris^*(M(h;i)) \otimes T_\cris^*(M(h';i')) \to T_\cris^*(M(h_s,i'')).$$  It suffices to find elements $u \in T_\cris^*(M(h;i))$ and $u' \in T_\cris^*(M(h';i'))$ such that $\alpha_s(u \otimes v) \neq 0$.  This is a question which can be approached using the techniques of \cite[\S5]{fl82}: by picking $u$ and $u'$ by finding integral elements in $\overline{K}$ satisfying certain relations (where $K$ still denotes is the $p$-adic field whose Galois representations we are investigating).  Like much of \cite[\S5]{fl82}, this boils down to a very technical argument about solving certain explicit equations.

\newcommand{\hata}{\widehat{a}}

We will now completely adopt the notation of \cite[\S5]{fl82}, especially of \S5.12, and illustrate this argument for the embedding $M(h_s;i'') \to M(\on{lcm}(h,h'),i'') \to M(h;i) \otimes M(h';i')$ given by
\[
 e_j \mapsto e_j + e_{j+ h_s}+ \ldots  + e_{j + (d_s-1) h_s}  \mapsto \sum_{m \in \ZZ/ d_s \ZZ} e_{j+ m h_s} \otimes e'_{j+m h_s + s}
\]
where $d_s h_s = \on{lcm}(h,h')$.
Recall $\alpha_s$ is the resulting map 
\[
\underline{U}_S(M(h;i)) \otimes \underline{U}_S(M(h';i')) \to \underline{U}_S(M(h_s;i'')).
\]
Elements of 
$\underline{U}_S(M(h_s;i'')) $
correspond to solutions to solutions $(\widehat{5})$ in \cite[\S5]{fl82}.  Let $\hata_j$ and $\hata_j'$ denote solutions to $(\widehat{5})$ corresponding to elements $u \in \underline{U}_S(M(h;i))$ and $u' \in \underline{U}_S(M(h';i')) $ respectively.  Since the $a_j$ are the constant terms of the elements $u_m$ in equation $(1)$, the solutions $\hat{t}_j$ to $(\widehat{5})$ corresponding to $\alpha_s(u \otimes u')$ are given by
\[
 \hat{t}_j = \sum_{m \in \ZZ/ d_s \ZZ} \hata_{j+m h_s} \cdot \hata'_{j+m h_s + s}
\]
By the proof of \cite[Lemme 5.12]{fl82}, 
\[
 \hata_{0}^{q^h} = \pi^\mu \hata_{0}
\]
for some explicit $\mu$, and $\hata_0$ determines all of the $\hata_j$.  In particular, each $\hata_j$ is a computable (rational) power of $\pi$ times a $(q^h-1)$-th root of unity.  If $\zeta$ is the root of unity appearing in $\hata_{0}$, then $\zeta^{q^j}$ is the root of unity appearing in $\hata_{h_s-j}$.  There is an analogous description for $\hata'_j$ (with root of unity $\zeta'$ for $\hata'_{s}$) and $\hat{t}_j$.  We will choose $\zeta$ and $\zeta'$ to force $\hat{t}_j \neq 0$ for all $j$ (equivalently any $j$).

Let $b_j = \hata_j \hata_{j+s}'$.  Then since $\hata_j$ and $\hata'_j$ satisfy $(\hat{5})$, we see that
\[
 b_j^q = \pi^{i''_j} b_{j-1}.
\]
We compute that
\[
 b_j^{q^{m h_s}} = \pi^{\mu_j (1 + q^{h_s} + \ldots + q^{(m-1) h_s})} b_{j-m h_s} \quad \text{ where } \quad \mu_j := i''_j q^{h_s-1} + i''_{j-1} q^{h_s-2} + \ldots + i''_{j-h_s+1}
\]
since $i''$ is periodic with period $h_s$.  Thus for any $m$, $b_{j + m h_s}$ is a solution to
\[
 X^{q^{d_s h_s}} = \pi^{\mu_j (1 + q^{h_s} + \ldots + q^{(d_s-1) h_s})} X
\]
and hence we may write $b_{j + m h_s}$ as a $(q^{\on{lcm}(h,h')}-1)$-th root of unity $\theta_m$ times a specified power $\pi^{\beta_j}$ with $\beta_j \in \QQ$.  Note that $\theta_{m} = \theta_{m+1}^{q^{h_s}}$.  In particular, we see that
\[
 \hat{t}_j = \sum_{m \in \ZZ/ d_s \ZZ} b_{j + m h_s} = (\theta_{d_s-1} + \ldots + \theta_{d_s-1}^{q^{h_s (d_s-1)}} ) \pi^{ \beta_j}.
\]
Now take $j=h_s$: we see that root of unity piece of $b_{h_s + h_s (d_s-1)} = b_0 = \hata_0 \hata'_{s}$ is $\zeta \zeta'$, so $\theta_{d_s-1} = \zeta \zeta'$.  
Thus to arrange that $\hat{t}_j \neq 0$, we must find $\zeta$ and $\zeta'$ so that $\theta_0 + \ldots + \theta_0^{q^{h_s (d_s-1)}}$ is non-zero, or equivalently (since it is a $(q^{\on{lcm}(h,h')}-1)$-th root of unity) that it does not reduce to $0$ in $\FF_{q^{\on{lcm}(h,h')}}$.  

Any choice of $\zeta$ and $\zeta'$ determines all of the $\hata_j$ and $\hata'_j$ and hence all of the $b_j$ and $\theta_m$.  Consider the polynomial
\[
 P(X) = X + X^{q^{h_s}} + \ldots + X^{q^{(d_s-1) h_s}} \in \FF_{q^{\on{lcm}(h,h')}}[X].
\]
For degree reasons, it cannot vanish on all of $\FF_{q^{\on{lcm}(h,h')}}$.  Since the natural multiplication map
\[
 \FF_{q^h} \tensor{\FF_q} \FF_{q^{h'}} \to \FF_{q^{\on{lcm}(h,h')}}
\]
is surjective and $P$ is additive (since we are in characteristic $q$) there must be elements $\overline{\zeta} \in \FF_{q^h}^\times$ and $\overline{\zeta}' \in \FF_{q^{h'}}^\times$ such that $P(\overline{\zeta} \cdot  \overline{\zeta}') \neq 0$.  Then taking $\zeta$ and $\zeta'$ to be the Teichmuller lifts of $\overline{\zeta}$ and $\overline{\zeta}'$, we obtain the desired solutions to $(\hat{5})$ of \cite{fl82} which show that $\alpha_s$ is non-zero.


\def\cftil#1{\ifmmode\setbox7\hbox{$\accent"5E#1$}\else
  \setbox7\hbox{\accent"5E#1}\penalty 10000\relax\fi\raise 1\ht7
  \hbox{\lower1.15ex\hbox to 1\wd7{\hss\accent"7E\hss}}\penalty 10000
  \hskip-1\wd7\penalty 10000\box7} \def\cprime{$'$} \def\cprime{$'$}
\providecommand{\bysame}{\leavevmode\hbox to3em{\hrulefill}\thinspace}
\providecommand{\MR}{\relax\ifhmode\unskip\space\fi MR }
\providecommand{\MRhref}[2]{%
  \href{http://www.ams.org/mathscinet-getitem?mr=#1}{#2}
}
\providecommand{\href}[2]{#2}

\end{document}